\documentclass{amsart}

\usepackage{amssymb}
\usepackage{amsmath}
\usepackage{amsthm}
\usepackage{amsfonts}

\usepackage{a4wide}

\newtheorem{proposition}{Proposition}[section]
\newtheorem{theorem}{Theorem}[section]
\newtheorem{corollary}{Corollary}[section]

\theoremstyle{definition}
\newtheorem{definition}{Definition}[section]
\newtheorem{remark}{Remark}[section]
\newtheorem{experiment}{Experiment}[section]
\newtheorem{example}{Example}[section]
\newtheorem{problem}{Problem}[section]

\numberwithin{equation}{section}

\begin{document}

\title{Recent progress in determining \(p\)-class field towers}

\author{Daniel C. Mayer}
\address{Naglergasse 53, 8010 Graz, Austria}
\email{algebraic.number.theory@algebra.at}
\urladdr{http://www.algebra.at}
\thanks{Research supported by the Austrian Science Fund (FWF): P 26008-N25}


\date{May 31, 2016}

\begin{abstract}
For a fixed prime \(p\),
the \(p\)-class tower \(\mathrm{F}_p^\infty{K}\) of a number field \(K\)
is considered to be known
if a pro-\(p\) presentation of the Galois group \(G=\mathrm{Gal}(\mathrm{F}_p^\infty{K}\vert K)\) is given.
In the last few years,
it turned out that the Artin pattern \(\mathrm{AP}(K)=(\tau(K),\varkappa(K))\)
consisting of targets \(\tau(K)=(\mathrm{Cl}_p{L})\) and kernels \(\varkappa(K)=(\ker{J_{L\vert K}})\)
of class extensions \(J_{L\vert K}:\mathrm{Cl}_p{K}\to\mathrm{Cl}_p{L}\) to
unramified abelian subfields \(L\vert K\) of the Hilbert \(p\)-class field \(\mathrm{F}_p^1{K}\)
only suffices for determining the two-stage approximation \(\mathfrak{G}=G/G^{\prime\prime}\) of \(G\).
Additional techniques had to be developed for identifying the group \(G\) itself:
searching strategies in descendant trees of finite \(p\)-groups,
iterated and multilayered IPADs of second order,
and the cohomological concept of Shafarevich covers involving relation ranks.
This enabled the discovery of three-stage towers of \(p\)-class fields
over quadratic base fields \(K=\mathbb{Q}(\sqrt{d})\)
for \(p\in\lbrace 2,3,5\rbrace\).
These non-metabelian towers reveal the new phenomenon of various tree topologies expressing the
mutual location of the groups \(G\) and \(\mathfrak{G}\).
\end{abstract}

\maketitle



\section{Introduction}
\label{s:Intro}

The reasons why our recent progress in determining \(p\)-class field towers
\cite{BuMa,Ma10,Ma12}
became possible during the past four years is due,
firstly, to a few crucial theoretical results by Artin
\cite{Ar1,Ar2}
and Shafarevich
\cite{Sh},
secondly, to actual implementations of group theoretic, resp. class field theoretic, algorithms by Newman
\cite{Nm1}
and O'Brien
\cite{OB},
resp. Fieker
\cite{Fi},
and finally, to several striking phenomena discovered by ourselves
\cite{Ma1,Ma3,Ma6,Ma7,Ma9,Ma11},
partially inspired by Bartholdi
\cite{BaBu},
Boston, Leedham-Green, Hajir
\cite{BoLG,BBH},
Bush
\cite{Bu},
and Nover
\cite{BoNo,No}.
In chronological order, these indispensable foundations can be summarized as follows.



\subsection{Class extension and transfer}
\label{ss:ClassExtAndTransfer}


Let \(p\) be a prime number and suppose that \(K\) is a number field
with non-trivial \(p\)-class group \(\mathrm{Cl}_p{K}:=\mathrm{Syl}_p\mathrm{Cl}_K>1\).
Then \(K\) possesses unramified abelian extensions \(L\vert K\) of relative degree a power of \(p\),
the biggest of them being the Hilbert \(p\)-class field \(\mathrm{F}_p^1{K}\) of \(K\).
For each of the extensions \(L\vert K\),
let \(J_{L\vert K}:\,\mathrm{Cl}_p{K}\to\mathrm{Cl}_p{L}\) be the \textit{class extension} homomorphism.

Artin used his reciprocity law of class field theory
\cite{Ar1}
for translating the arithmetical properties of \(J_{L\vert K}\),
the \(p\)-capitulation kernel \(\ker{J_{L\vert K}}\) and the target \(p\)-class group \(\mathrm{Cl}_p{L}\),
into group theoretic properties of the \textit{transfer} homomorphism
\(T_{\mathfrak{G},H}:\,\mathfrak{G}\to H/H^\prime\)
from the Galois group \(\mathfrak{G}:=\mathrm{Gal}(\mathrm{F}_p^2{K}\vert K)\)
of the second Hilbert \(p\)-class field \(\mathrm{F}_p^2{K}=\mathrm{F}_p^1{\mathrm{F}_p^1{K}}\)
to the abelianization of the subgroup \(H:=\mathrm{Gal}(\mathrm{F}_p^2{K}\vert L)\).
The reciprocity map establishes an isomorphism between the targets,
\(H/H^\prime\simeq\mathrm{Cl}_p{L}\),
and an isomorphism between the domains,
\(\mathfrak{G}/\mathfrak{G}^\prime\simeq\mathrm{Cl}_p{K}\),
in particular, between the kernels \(\ker{\tilde{T}_{\mathfrak{G},H}}\simeq\ker{J_{L\vert K}}\),
of the \textit{induced} transfer \(\tilde{T}_{\mathfrak{G},H}:\,\mathfrak{G}/\mathfrak{G}^\prime\to H/H^\prime\) and \(J_{L\vert K}\)
\cite{Ar2}.
In \S\
\ref{ss:ArtinPattern},
we introduce the Artin pattern \(\mathrm{AP}\) of \(K\), resp. \(\mathfrak{G}\),
as the collection of all targets and kernels of the homomorphisms \(J_{L\vert K}\), resp. \(\tilde{T}_{\mathfrak{G},H}\),
where \(L\) varies over intermediate fields \(K\le L\le \mathrm{F}_p^1{K}\),
resp. \(H\) varies over intermediate groups \(\mathfrak{G}^\prime\le H\le\mathfrak{G}\).
The Artin pattern has turned out to be sufficient for identifying
a finite batch of candidates for
the \textit{second \(p\)-class group} \(\mathfrak{G}\) of \(K\),
frequently even a unique candidate.



\subsection{Relation rank of the \(p\)-tower group}
\label{ss:RelationRank} 

An invaluably precious aid in identifying the \(p\)-\textit{tower group},
that is the Galois group \(G:=\mathrm{Gal}(\mathrm{F}_p^\infty{K}\vert K)\)
of the maximal unramified pro-\(p\) extension \(\mathrm{F}_p^\infty{K}\) of a number field \(K\),
has been elaborated by Shafarevich
\cite{Sh},
who determined bounds \(\varrho\le d_2{G}\le\varrho+r+\theta\)
for the \textit{relation rank} \(d_2{G}:=\dim_{\mathbb{F}_p}\mathrm{H}^2\left(G,\mathbb{F}_p\right)\) of \(G\)
in terms of the \(p\)-class rank \(\varrho\) of \(K\),
the torsion free Dirichlet unit rank \(r=r_1+r_2-1\) of a number field \(K\) with signature \((r_1,r_2)\),
and the invariant \(\theta\) which takes the value \(1\), if \(K\) contains a primitive \(p\)th root of unity,
and \(0\), otherwise.

The derived length of the group \(G\) is called
the \textit{length} \(\ell_p{K}=\mathrm{dl}(G)\) of the \(p\)-class tower of \(K\).
The metabelianization \(G/G^{\prime\prime}\) of the \(p\)-tower group \(G\)
is isomorphic to the second \(p\)-class group \(\mathfrak{G}\) of \(K\) in \S\
\ref{ss:ClassExtAndTransfer},
which can be viewed as a two-stage approximation of \(G\).



\subsection{Cover and Shafarevich cover}
\label{ss:Cover}

Let \(p\) be a prime and
\(\mathfrak{G}\) be a finite \textit{metabelian} \(p\)-group.

\begin{definition}
\label{dfn:Cover}

By the \textit{cover} of \(\mathfrak{G}\) we understand
the set of all (isomorphism classes of) finite \(p\)-groups
whose second derived quotient is isomorphic to \(\mathfrak{G}\),
\[\mathrm{cov}(\mathfrak{G}):=\left\lbrace G\mid \mathrm{ord}(G)<\infty,\ G/G^{\prime\prime}\simeq\mathfrak{G}\right\rbrace.\]
By eliminating the finiteness condition, we obtain the \textit{complete cover} of \(\mathfrak{G}\),
\[\mathrm{cov}_c(\mathfrak{G}):=\left\lbrace G\mid G/G^{\prime\prime}\simeq\mathfrak{G}\right\rbrace.\]

\end{definition}


\begin{remark}
\label{rmk:Cover}
The unique metabelian element of \(\mathrm{cov}(\mathfrak{G})\) is
the isomorphism class of \(\mathfrak{G}\) itself.
\end{remark}



\begin{theorem}
\label{thm:RelationRank}

(Shafarevich \(1964\))

Let \(p\) be a prime number
and denote by \(\zeta\) a primitive \(p\)th root of unity.
Let \(K\) be a number field with signature \((r_1,r_2)\) and
torsionfree Dirichlet unit rank \(r=r_1+r_2-1\),
and let \(S\) be a finite set of non-archimedean or real archimedean places of \(K\).
Assume that no place in \(S\) divides \(p\).

Then the relation rank \(d_2{G_S}:=\dim_{\mathbb{F}_p}H^2(G_S,\mathbb{F}_p)\) of
the Galois group \(G_S:=\mathrm{Gal}(K_S\vert K)\) of
the maximal pro-\(p\) extension \(K_S\) of \(K\)
which is unramified outside of \(S\)
is bounded from above by

\begin{equation}
\label{eqn:RelationRank}
d_2{G_S}\le
\begin{cases}
d_1{G_S}+r   & \text{ if } S\ne\emptyset \text { or } \zeta\notin K, \\
d_1{G_S}+r+1 & \text{ if } S=\emptyset  \text { and } \zeta\in K,
\end{cases}
\end{equation}

\noindent
where \(d_1{G_S}:=\dim_{\mathbb{F}_p}H^1(G_S,\mathbb{F}_p)\) denotes the generator rank of \(G_S\).

\end{theorem}

\begin{proof}
The original statement in
\cite[Thm. 6, \((18^\prime)\)]{Sh}
contained a serious misprint
which was corrected in
\cite[Thm. 5.5, p. 28]{Ma10}.
\end{proof}



\begin{definition}
\label{dfn:ShafarevichCover}
Let \(p\) be a prime and
\(K\) be a number field with
\(p\)-class rank \(\varrho:=d_1(\mathrm{Cl}_p{K})\),
torsionfree Dirichlet unit rank \(r\),
and second \(p\)-class group \(\mathfrak{G}:=\mathrm{G}_p^2 K\).
By the \textit{Shafarevich cover},
\(\mathrm{cov}(\mathfrak{G},K)\),
of \(\mathfrak{G}\) \textit{with respect to} \(K\)
we understand the subset of \(\mathrm{cov}(\mathfrak{G})\) whose elements \(G\)
satisfy the following condition for their relation rank \(d_2{G}\):

\begin{equation}
\label{eqn:ShafarevichCover}
\varrho\le d_2{G}\le\varrho+r+\theta,
\quad \text{ where } \quad
\theta:=
\begin{cases}
1 & \text{ if } K \text{ contains the } p\text{th roots of unity,} \\
0 & \text{ otherwise.} 
\end{cases}
\end{equation}

\end{definition}


\begin{definition}
\label{dfn:SigmaGroup}

A finite \(p\)-group or an infinite pro-\(p\) group \(G\), with a prime number \(p\ge 2\),
is called a \(\sigma\)-\textit{group},
if it possesses a \textit{generator inverting} (GI-)automorphism \(\sigma\in\mathrm{Aut}(G)\)
which acts as the inversion mapping on the derived quotient \(G/G^\prime\), that is,

\begin{equation}
\label{eqn:SigmaGroup}
\sigma(g)G^\prime=g^{-1}G^\prime \quad \text{ for all } g\in G.
\end{equation}

\noindent
\(G\) is called a \textit{Schur \(\sigma\)-group} if it is a \(\sigma\)-group
with balanced presentation \(d_2{G}=d_1{G}\).

\end{definition}



\subsection{\(p\)-Group generation algorithm}
\label{ss:pGroupAlgorithm}

The descendant tree \(\mathcal{T}(R)\) of a finite \(p\)-group \(R\)
\cite{Ma6}
can be constructed recursively
by starting at the root \(R\)
and successively determining immediate descendants
by iterated executions of the \(p\)-\textit{group generation algorithm}
\cite{HEO},
which was designed by Newman
\cite{Nm1},
implemented for \(p\in\lbrace 2,3\rbrace\) and \(R=C_p\times C_p\) by Ascione and collaborators
\cite{AHL,As1},
and implemented in full generality for GAP
\cite{GAP}
and MAGMA
\cite{BCP,BCFS,MAGMA}
by O'Brien
\cite{OB}.



\subsection{Construction of unramified abelian \(p\)-extensions}
\label{ss:pExtensions}

Routines for constructing all intermediate fields \(K\le L\le\mathrm{F}_{(c)}{K}\)
between a number field \(K\) and the ray class field \(\mathrm{F}_{(c)}{K}\)
modulo a given conductor \(c\) of \(K\)
have been implemented in MAGMA
\cite{MAGMA}
by Fieker
\cite{Fi}.
Here, we shall use this class field package
for finding unramified cyclic \(p\)-extensions \(L\vert K\) with conductor \(c=1\) only.
These fields are located between \(K\) and its Hilbert \(p\)-class field \(\mathrm{F}_p^1{K}\).



\subsection{The Artin pattern}
\label{ss:ArtinPattern}

Let \(p\) be a fixed prime
and \(K\) be a number field with \(p\)-class group \(\mathrm{Cl}_p{K}\) of order \(p^v\),
where \(v\ge 0\) denotes a non-negative integer.

\begin{definition}
\label{dfn:FieldLayers}

For each integer \(0\le n\le v\), the system
\(\mathrm{Lyr}_n{K}:=\lbrace K\le L\le\mathrm{F}_p^1{K}\mid \lbrack L:K\rbrack=p^n\rbrace\)
is called the \textit{\(n\)th layer} of abelian unramified \(p\)-extensions of \(K\).

\end{definition}


\begin{definition}
\label{dfn:FieldTTTandTKT}

For each intermediate field \(K\le L\le\mathrm{F}_p^1{K}\),
let \(J_{L\vert K}:\,\mathrm{Cl}_p{K}\to\mathrm{Cl}_p{L}\) be the \textit{class extension},
which can also be called the number theoretic \textit{transfer} from \(K\) to \(L\).

\begin{enumerate}

\item
Let
\(\tau(K):=\lbrack\tau_0{K};\ldots;\tau_v{K}\rbrack\)
be the \textit{multi-layered transfer target type} (TTT) of \(K\), where
\(\tau_n{K}:=(\mathrm{Cl}_p{L})_{L\in\mathrm{Lyr}_n{K}}\) for each \(0\le n\le v\).

\item
Let
\(\varkappa(K):=\lbrack\varkappa_0{G};\ldots;\varkappa_v{G}\rbrack\)
be the \textit{multi-layered transfer kernel type} (TKT) or
\textit{multi-layered \(p\)-capitulation type} of \(K\), where
\(\varkappa_n{K}:=(\ker{J_{L\vert K}})_{L\in\mathrm{Lyr}_n{K}}\) for each \(0\le n\le v\).

\end{enumerate}

\end{definition}


\begin{definition}
\label{dfn:FieldArtinPattern}

The pair
\(\mathrm{AP}(K):=(\tau(K),\varkappa(K))\)
is called the restricted \textit{Artin pattern} of \(K\).

\end{definition}


Let \(p\) be a prime number
and \(G\) be a pro-\(p\) group with finite abelianization \(G/G^\prime\),
more precisely, assume that the commutator subgroup \(G^\prime\)
is of index \((G:G^\prime)=p^v\) with an integer exponent \(v\ge 0\).

\begin{definition}
\label{dfn:GroupLayers}

For each integer \(0\le n\le v\), let
\(\mathrm{Lyr}_n{G}:=\lbrace G^\prime\le H\le G\mid (G:H)=p^n\rbrace\)
be the \textit{\(n\)th layer} of normal subgroups of \(G\) containing \(G^\prime\).

\end{definition}


\begin{definition}
\label{dfn:GroupTTTandTKT}

For any intermediate group \(G^\prime\le H\le G\), we denote by
\(T_{G,H}:\,G\to H/H^\prime\)
the \textit{Artin transfer} homomorphism from \(G\) to \(H/H^\prime\)
\cite[Dfn. 3.1]{Ma9},
and by \(\tilde{T}_{G,H}:\,G/G^\prime\to H/H^\prime\) the \textit{induced transfer}.

\begin{enumerate}

\item
Let
\(\tau(G):=\lbrack\tau_0{G};\ldots;\tau_v{G}\rbrack\)
be the \textit{multi-layered transfer target type} (TTT) of \(G\), where
\(\tau_n{G}:=(H/H^\prime)_{H\in\mathrm{Lyr}_n{G}}\) for each \(0\le n\le v\).

\item
Let
\(\varkappa(G):=\lbrack\varkappa_0{G};\ldots;\varkappa_v{G}\rbrack\)
be the \textit{multi-layered transfer kernel type} (TKT) of \(G\), where
\(\varkappa_n{G}:=(\ker{\tilde{T}_{G,H}})_{H\in\mathrm{Lyr}_n{G}}\) for each \(0\le n\le v\).

\end{enumerate}

\end{definition}


\begin{definition}
\label{dfn:GroupArtinPattern}

The pair
\(\mathrm{AP}(G):=(\tau(G),\varkappa(G))\)
is called the restricted \textit{Artin pattern} of \(G\).

\end{definition}



\begin{theorem}
\label{thm:ArtinPattern}
Let \(p\) be a prime number.
Assume that \(K\) is a number field, and let
\(\mathfrak{G}:=\mathrm{G}_p^2{K}\) be the second \(p\)-class group of \(K\).
Then \(\mathfrak{G}\) and \(K\) share a common restricted Artin pattern,

\begin{equation}
\label{eqn:ArtinPattern}
\mathrm{AP}(\mathfrak{G})=\mathrm{AP}(K), \quad \text{ that is } \quad \tau(\mathfrak{G})=\tau(K)\text{ and } \varkappa(\mathfrak{G})=\varkappa(K),
\end{equation}

\noindent
in the sense of componentwise isomorphisms.

\end{theorem}

\begin{proof}
A sketch of the proof is indicated in
\cite{My},
\cite[\S\ 2.3, pp. 476--478]{Ma2}
and
\cite[Thm. 1.1, p. 402]{Ma4},
but the precise proof has been given by Hasse in
\cite[\S\ 27, pp. 164--175]{Ha2}.
\end{proof}



\begin{theorem}
\label{thm:RstrAPofCompleteCover}

Let \(\mathfrak{G}\) be a finite metabelian \(p\)-group.
Then all elements of the complete cover of \(\mathfrak{G}\) share a common restricted Artin pattern:

\begin{equation}
\label{eqn:RstrAPofCompleteCover}
 \mathrm{AP}(G)=\mathrm{AP}(\mathfrak{G}), \quad \text{ for all } \quad G\in\mathrm{cov}_c(\mathfrak{G}).
\end{equation}

\end{theorem}

\begin{proof}
This is the Main Theorem of
\cite[Thm. 5.4, p. 86]{Ma9}.
\end{proof}



\begin{definition}
\label{dfn:FieldIPADandIPOD}

The first order approximation
\(\tau^{(1)}{K}:=\lbrack\tau_0{K};\tau_1{K}\rbrack\)
of the TTT, resp.
\(\varkappa^{(1)}{K}:=\lbrack\varkappa_0{K};\varkappa_1{K}\rbrack\)
of the TKT,
is called the \textit{index-\(p\) abelianization data} (IPAD),
resp. \textit{index-\(p\) obstruction data} (IPOD), of \(K\).

\end{definition}


\begin{definition}
\label{dfn:GroupIPADandIPOD}

The first order approximation
\(\tau^{(1)}{G}:=\lbrack\tau_0{G};\tau_1{G}\rbrack\)
of the TTT, resp.
\(\varkappa^{(1)}{G}:=\lbrack\varkappa_0{G};\varkappa_1{G}\rbrack\)
of the TKT,
is called the \textit{index-\(p\) abelianization data} (IPAD),
resp. \textit{index-\(p\) obstruction data} (IPOD), of \(G\).

\end{definition}


As mentioned in \S\
\ref{ss:ClassExtAndTransfer},
the TTT and TKT, frequently even the IPAD and IPOD,
are sufficient for identifying the second \(p\)-class group
\(\mathfrak{G}=\mathrm{G}_p^2{K}\) of a number field \(K\).
This was discovered by ourselves in
\cite{Ma1,Ma3,Ma4},
and independently by Boston and collaborators
\cite{BoLG,BBH}.

For finding the \(p\)-class tower group \(G=\mathrm{G}_p^\infty{K}\), however,
we need the following non-abelian generalization,
which requires computing extensions of relative degree \(p^2\) instead of \(p\) over \(K\),
and was introduced by ourselves in
\cite{Ma7,Ma10,Ma11}
and by Bartholdi, Bush, and Nover in
\cite{Bu,BaBu,BoNo,No}.

\begin{definition}
\label{dfn:FieldIteratedIPAD}

\(\tau^{(2)}{K}:=\lbrack\tau_0{K};(\tau^{(1)}{L})_{L\in\mathrm{Lyr}_1{K}}\rbrack\)
is called \textit{iterated} IPAD \textit{of second order} of \(K\).

\end{definition}


\begin{definition}
\label{dfn:GroupIteratedIPAD}

\(\tau^{(2)}{G}:=\lbrack\tau_0{G};(\tau^{(1)}{H})_{H\in\mathrm{Lyr}_1{G}}\rbrack\)
is called \textit{iterated} IPAD \textit{of second order} of \(G\).

\end{definition}


\begin{theorem}
\label{thm:IteratedIPAD}
Let \(p\) be a prime number.
Assume that \(K\) is a number field with \(p\)-class tower group
\(G:=\mathrm{G}_p^\infty{K}\).
Then \(G\) and \(K\) share a common iterated IPAD of second order,

\begin{equation}
\label{eqn:IteratedIPAD}
\tau^{(2)}{G}=\tau^{(2)}{K},
\end{equation}

\noindent
in the sense of componentwise isomorphisms.

\end{theorem}

\begin{proof}
In Theorem
\ref{thm:ArtinPattern},
we proved that \(\tau(\mathfrak{G})=\tau(K)\), and from Theorem
\ref{thm:RstrAPofCompleteCover}
we know that \(\tau(G)=\tau(\mathfrak{G})\).
Thus we have, in particular,
\(\tau^{(1)}{G}=\lbrack\tau_0{G};\tau_1{G}\rbrack=\lbrack\tau_0{K};\tau_1{K}\rbrack=\tau^{(1)}{K}\),
where \(\tau_0{G}=G/G^\prime\simeq\mathrm{Cl}_p{K}=\tau_0{K}\) and
\(\tau_1{G}=(\tau_0{H})_{H\in\mathrm{Lyr}_1{G}}=(\tau_0{L})_{L\in\mathrm{Lyr}_1{K}}=\tau_1{K}\),
i.e., \(\tau_0{H}=H/H^\prime\simeq\mathrm{Cl}_p{L}=\tau_0{L}\),
for all \(H\in\mathrm{Lyr}_1{G}\) such that \(H=\mathrm{Gal}(\mathrm{F}_p^\infty{K}\vert L)\) with \(L\in\mathrm{Lyr}_1{K}\).

It remains to show that \(\tau_1{H}=(\tau_0{U})_{U\in\mathrm{Lyr}_1{H}}=(\tau_0{M})_{M\in\mathrm{Lyr}_1{L}}=\tau_1{L}\).
Let \(U\in\mathrm{Lyr}_1{H}\), then \(U=\mathrm{Gal}(\mathrm{F}_p^\infty{K}\vert M)\) for some \(M\in\mathrm{Lyr}_1{L}\),
since \(G^{\prime\prime\prime}\unlhd H^{\prime\prime}\unlhd U^\prime\unlhd H^\prime\unlhd U\unlhd H\) and
\(p=(H:U)=\#(H/U)=\#\mathrm{Gal}(M\vert L)=\lbrack M:L\rbrack\).
Since \(U^\prime\) is the smallest subgroup of \(U\) with abelian quotient \(U/U^\prime\),
we must have \(U^\prime=\mathrm{Gal}(\mathrm{F}_p^\infty{K}\vert\mathrm{F}_p^1{M})\) and thus
\(\tau_0{U}=U/U^\prime\simeq\mathrm{Gal}(\mathrm{F}_p^1{M}\vert M)\simeq\mathrm{Cl}_p{M}=\tau_0{M}\), as required.
Observe that, in general, neither \(U\unlhd G\) nor \(G^{\prime\prime}\le U^\prime\),
and thus \(G=\mathrm{G}_p^\infty{K}\) cannot be replaced by \(\mathfrak{G}=\mathrm{G}_p^2{K}\).
We could, however, take \(\mathrm{G}_p^3{K}\) instead of \(G\).
\end{proof}


\begin{remark}
\label{rmk:ArtinPattern}

The TKT and the IPOD contain some standard information which can be omitted.

\begin{enumerate}

\item
Since the zeroth layer (top layer), \(\mathrm{Lyr}_0{G}=\lbrace G\rbrace\),
consists of the group \(G\) alone,
and \(T_{G,G}:\,G\to G/G^\prime\) is the natural projection onto the commutator quotient
with kernel \(\ker{T_{G,G}}=G^\prime\), resp. \(\ker{\tilde{T}_{G,G}}=G^\prime/G^\prime\simeq 1\),
we usually omit the trivial top layer \(\varkappa_0{G}=\lbrace 1\rbrace\)
and identify the IPOD \(\varkappa^{(1)}{G}\) with the first layer \(\varkappa_1{G}\) of the TKT.

\item
In the case of an elementary abelianization of rank two, \((G:G^\prime)=p^2\),
we also identify the TKT \(\varkappa(G)\) with its first layer \(\varkappa_1{G}\),
since the second layer (bottom layer), \(\mathrm{Lyr}_2{G}=\lbrace G^\prime\rbrace\),
consists of the commutator subgroup \(G^\prime\) alone,
and the kernel of \(T_{G,G^\prime}:\,G\to G^\prime/G^{\prime\prime}\)
is always \textit{total}, that is \(\ker{T_{G,G^\prime}}=G\), resp. \(\ker{\tilde{T}_{G,G^\prime}}=G/G^\prime\),
according to the \textit{principal ideal theorem}
\cite{Fw}.

\end{enumerate}

\end{remark}



\subsection{Monotony on descendant trees}
\label{ss:MonotonyDescendantTrees}

\begin{definition}
\label{dfn:DescTree}

Let \(p\) be a prime 
and \(G\), \(H\) and \(R\) be finite \(p\)-groups.

The \textit{lower central series} (LCS) \((\gamma_n{G})_{n\ge 1}\) of \(G\) is defined recursively by
\(\gamma_1{G}:=G\), and \(\gamma_n{G}:=\left\lbrack\gamma_{n-1}{G},G\right\rbrack\) for \(n\ge 2\).

We call \(G\) an \textit{immediate descendant} (or \textit{child}) of \(H\),
and \(H\) \textit{the parent} of \(G\),
if \(H\simeq G/\gamma_c{G}\) is isomorphic to the
biggest non-trivial lower central quotient of \(G\), that is,
to the image of the natural projection \(\pi:\,G\to G/\gamma_c{G}\) of \(G\) onto the quotient by
the last non-trivial term \(\gamma_c{G}>1\) of the LCS of \(G\),
where \(c:=\mathrm{cl}(G)\) denotes the \textit{nilpotency class} of \(G\).
In this case, we consider the projection \(\pi\)
as a directed edge \(G\to H\) from \(G\) to \(H\simeq\pi{G}\),
and we speak about the \textit{parent operator},
\(\pi:\,G\to \pi{G}=G/\gamma_c{G}\simeq H\).

We call \(G\) a \textit{descendant} of \(H\),
and \(H\) an \textit{ancestor} of \(G\),
if there exists a finite \textit{path} of directed edges
\(\left(Q_j\to Q_{j+1}\right)_{0\le j<\ell}\) such that \(G=Q_0\), \(Q_{j+1}=\pi{Q_j}\) for \(0\le j<\ell\), and \(H=Q_\ell\), that is,
\(G=Q_0\to Q_1\to\ldots\to Q_{\ell-1}\to Q_\ell=H\),
where \(\ell\ge 0\) denotes the \textit{path length}.

The \textit{descendant tree} of \(R\), denoted by \(\mathcal{T}(R)\),
is the rooted directed tree with root \(R\)
having the isomorphism classes of all descendants of \(R\) as its \textit{vertices}
and all (child, parent)-pairs \((G,H)\) among the descendants \(G,H\) of \(R\)
as its \textit{directed edges} \(G\to\pi{G}\simeq H\).
By means of formal iterations \(\pi^j\) of the parent operator \(\pi\),
each vertex of the descendant tree \(\mathcal{T}(R)\) can be connected with \(R\)
by a finite path of edges:
\(\mathcal{T}(R)=\left\lbrace G\mid G=\pi^0{G}\to\pi^1{G}\to\pi^2{G}\to\ldots\to\pi^\ell{G}=R \text{ for some } \ell\ge 0\right\rbrace\).

\end{definition}


\begin{theorem}
\label{thm:Monotony}

Let \(\mathcal{T}(R)\) be the descendant tree with root \(R>1\), a finite non-trivial \(p\)-group,
and let \(G\to\pi{G}\) be a directed edge of the tree.
Then the restricted Artin pattern \(\mathrm{AP}=(\tau,\varkappa)\)
satisfies the following monotonicity relations

\begin{equation}
\label{eqn:Monotony}
\begin{aligned}
     \tau(G) &\ge \tau(\pi{G}), \\
\varkappa(G) &\le \varkappa(\pi{G}),
\end{aligned}
\end{equation}

\noindent
that is, the TTT \(\tau\) is an isotonic mapping
and the TKT \(\varkappa\) is an antitonic mapping
with respect to the partial order \(G>\pi{G}\)
induced by the directed edges \(G\to\pi{G}\).

\end{theorem}

\begin{proof}
This is Theorem 3.1 in
\cite{Ma12}.
\end{proof}

This result yields the crucial break-off condition for recursive executions
of the \(p\)-group generation algorithm,
when we want to find a finite \(p\)-group with assigned Artin pattern.



\subsection{State of research}
\label{ss:StateOfResearch}

The state of research on \(p\)-class field towers in the year \(2008\)
was summarized in a succinct form by McLeman
\cite{McL}.
He literally pointed out the following problem on p. \(200\) and p. \(205\) of his paper.

\begin{problem}
\label{prb:McLeman}
For odd primes \(p\),
there are no known examples of imaginary quadratic number fields
with \(p\)-class rank \(2\) and
either an infinite \(p\)-class tower
or a \(p\)-class tower of length bigger than \(2\).
For \(p=3\), the longest known finite towers are of length \(2\).
\end{problem}

McLeman's survey on the state of the art changed
when Bush and ourselves found the first imaginary quadratic fields having
\(3\)-class field towers of exact length \(3\)
\cite{BuMa}
on \(24\) August \(2012\).

The reason why McLeman formulated this open problem
for \textit{imaginary quadratic} fields with \(\varrho=2\)
is the well-known fact that \(\varrho=1\) implies an abelian single-stage tower,
and, on the other hand, the strong criterion by Koch and Venkov
\cite{KoVe}
that \(\varrho\ge 3\) enforces an infinite \(p\)-class tower
for odd primes \(p\ge 3\).
We shall come back to this criterion
in \S\
\ref{s:Infinite3Towers}
on infinite \(3\)-class towers.

However, until \(2012\), a much more extensive problem
for finite \(p\)-class towers
of any algebraic number field \(K\) was open.

\begin{problem}
\label{prb:2012}
No examples are known of number fields \(K\) having
a \(2\)-class tower of length \(\ell_2{K}\ge 4\) or
a \(3\)-class tower of length \(\ell_3{K}\ge 3\) or
a \(p\)-class tower of length \(\ell_p{K}\ge 2\) for \(p\ge 5\).
\end{problem}

Since the joint discovery with Bush
\cite{BuMa},
we unsuccessfully tried to
extend the result from \(p=3\) to \(p=5\) for nearly \(4\) years,
as documented in the historical introduction of
\cite{Ma12},
until a lucky coincidence of several unexpected facts
enabled a significant break-through on \(07\) April \(2016\):

\begin{enumerate}

\item
Due to a bug in earlier MAGMA versions,
the \(5\)-capitulation type \(\varkappa(K)\) of the real quadratic field \(K=\mathbb{Q}(\sqrt{d})\)
with discriminant \(d=3\,812\,377\) could not be computed
until MAGMA V2.21-11 was released on \(04\) April \(2016\).

\item
In the tree \(\mathcal{T}(\langle 5^2,2\rangle)\) of \(5\)-groups of coclass \(1\), 
the crucial bifurcation at the \(4\)th mainline vertex \(\langle 5^5,30\rangle\) was unknown up to now.
It gives rise to the candidates \(\langle 5^5,30\rangle-\#2;n\), \(n\in\lbrace 10,22,23,34,35\rbrace\), 
for the \(5\)-tower group \(\mathrm{G}_5^\infty{K}\) of \(K=\mathbb{Q}(\sqrt{3\,812\,377})\).

\item
Whereas the smallest non-metabelian \(p\)-tower groups \(G=\mathrm{G}_p^\infty{K}\), for \(p\ge 3\) an odd prime,
of imaginary quadratic fields \(K=\mathbb{Q}(\sqrt{d})\), \(d<0\), are of order \(\lvert G\rvert=p^8\) and coclass \(\mathrm{cc}(G)=2\),
those of real quadratic fields, \(d>0\), have order only \(\lvert G\rvert=p^7\), coclass \(\mathrm{cc}(G)=1\),
and do not require arithmetical computations of high complexity 
with extensions \(M\vert K\) of relative degree \(\lbrack M:K\rbrack=p^2\) for their justification.
IPADs of first order are sufficient.

\end{enumerate}



\subsection{Overview}
\label{ss:Overview}

To avoid any misinterpretations of our notation,
an essential remark must be made at the beginning:
Throughout this article, we use the \textit{logarithmic form}
of type invariants of finite abelian \(p\)-groups \(A\),
that is, we abbreviate the cumbersome power form of type invariants
\(\left(\overbrace{p^{e_1},\ldots,p^{e_1}}^{r_1 \text{ times}},\ldots,\overbrace{p^{e_n},\ldots,p^{e_n}}^{r_n \text{ times}}\right)\),
with strictly decreasing \(e_1>\ldots>e_n\), by writing
\(\left(e_1^{r_1},\ldots,e_n^{r_n}\right)\) 
with formal exponents \(r_i\ge 1\) denoting iteration.
If \(e_1<10\), 
which will always be the case in this paper,
then we even omit the separating commas, thus saving a lot of space.

The layout of this survey article is the following.
In \S\
\ref{s:ThreeStage5Towers}
we immediately celebrate our most recent sensational discovery
of the long desired three-stage towers of \(5\)-class fields.
We continue with a recall of the meanwhile well-known
three-stage towers of \(3\)-class fields in \S\
\ref{s:ThreeStage3Towers}
and of \(2\)-class fields in \S\
\ref{s:ThreeStage2Towers}.

In \S\
\ref{s:TreeTopologies}
we present the new phenomenon of tree topologies
expressing the mutual location of second and third \(p\)-class groups
on descendant trees of finite \(p\)-groups.
An important remark has to be made in \S\
\ref{s:BicyclicBiquadratic}
on the published form of the
Shafarevich theorem on the relation rank of the \(p\)-tower group
when the base field contains a primitive \(p\)th root of unity.

Although the focus will mainly be on the new phenomena of three-stage towers,
we also give, en passant in \S\
\ref{s:TwoStageTowers},
the first criteria for two-stage towers of \(p\)-class fields,
\(\ell_p{K}=2\), for \(p\in\lbrace 5,7\rbrace\),
independently of the base field \(K\).

In \S\
\ref{s:Infinite3Towers}
we consider \(3\)-class towers of
three complex quadratic fields \(K=\mathbb{Q}(\sqrt{d})\)
with \(3\)-class group \(\mathrm{Cl}_3{K}\) of type \((1^3)\),
which have infinite length \(\ell_3{K}=\infty\), according to
\cite{KoVe}.
We emphasize that we are far from having explicit pro-\(3\) presentations
of the \(3\)-tower groups \(G=\mathrm{G}_3^\infty{K}\) and we do not know
the rate of growth for the orders of the successive derived quotients
\(G/G^{(n)}\simeq\mathrm{G}_3^n{K}\), \(n\ge 2\).
Even for the second \(3\)-class groups \(\mathfrak{G}=\mathrm{G}_3^2{K}\),
we only have lower bounds for the orders.



\section{Three-stage towers of \(5\)-class fields}
\label{s:ThreeStage5Towers}

\begin{experiment}
\label{exp:5Capitulation}

As documented in \S\ 3.2.7, p. 427, of
\cite{Ma4},
we used the class field package by Fieker
\cite{Fi}
in the computational algebra system MAGMA
\cite{MAGMA}
for constructing the unramified cyclic quintic extensions \(L_i\vert K\), \(1\le i\le 6\),
of each of the \(377\) real quadratic fields \(K=\mathbb{Q}(\sqrt{d})\)
with discriminants \(0<d\le 26\,695\,193\)
and \(5\) class group of type \((1^2)\).
However, at that early stage, we only computed the first component \(\tau(K)\)
of the Artin pattern \(\mathrm{AP}(K)=(\tau(K),\varkappa(K))\),
since \(\tau(K)\) uniquely determines \(\varkappa(K)\) for the ground state of the TKTs a.\(2\) and a.\(3\),
according to Theorem 3.8 and Table 3.3 in
\cite[\S\ 3.2.5, pp. 423--424]{Ma4}.
In this manner, we were able to classify \(360=13+55+292\) cases with TKTs a.\(1\), a.\(2\), a.\(3\),
where the second \(5\)-class group \(\mathfrak{G}=\mathrm{G}_5^2{K}\) is of coclass \(\mathrm{cc}(\mathfrak{G})=1\),
listed in Table 3.4 of
\cite[\S\ 3.2.7, p. 427]{Ma4},
and to separate \(14\) cases with \(\mathrm{cc}(\mathfrak{G})=2\),
discussed in
\cite[\S\ 3.5.3, p. 449]{Ma4}.
There remained \(3\) cases of first excited states of the TKTs a.\(2\) and a.\(3\),
where the TTT 
\(\tau(K)=\lbrack 21^3,(1^2)^5\rbrack\)
is unable to distinguish between the TKTs.
As mentioned at the end of \S\
\ref{ss:StateOfResearch},
the first MAGMA version which admitted the computation of the TKTs for these difficult cases
was V2.21-11, released on \(04\) April \(2016\).
The result was TKT a.\(2\), \(\varkappa(K)=(1,0^5)\), for \(d\in\lbrace 3\,812\,377, 19\,621\,905\rbrace\)
and TKT a.\(3\), \(\varkappa(K)=(2,0^5)\), for \(d=21\,281\,673\).

\end{experiment}

After this initial number theoretic experiment with computational techniques of \S\
\ref{ss:pExtensions},
a translation from arithmetic to group theory with Artin's reciprocity law, described in \S\
\ref{ss:ClassExtAndTransfer},
maps the Artin pattern \(\mathrm{AP}(K)=(\tau(K),\varkappa(K))\) of the real quadratic fields \(K\)
to the Artin pattern \(\mathrm{AP}(\mathfrak{G})=(\tau(\mathfrak{G}),\varkappa(\mathfrak{G}))\)
of their second \(5\)-class groups \(\mathfrak{G}=\mathrm{G}_5^2{K}\),
which forms the input for the strategy of \textit{pattern recognition via Artin transfers}
by conducting a search for suitable finite \(5\)-groups \(\mathfrak{G}\)
having the prescribed Artin pattern \(\mathrm{AP}(\mathfrak{G})\).
This is done by recusive iterations of the \(p\)-group generation algorithm in \S\
\ref{ss:pGroupAlgorithm}
until a termination condition is satisfied, due to the monotony of Artin patterns on descendant trees in \S\
\ref{ss:MonotonyDescendantTrees}.

The reason why we decided to take the Artin pattern
\(\mathrm{AP}(\mathfrak{G})=(\tau(\mathfrak{G}),\varkappa(\mathfrak{G}))\),
with
\(\tau_0{\mathfrak{G}}=(1^2)\), \(\tau_1{\mathfrak{G}}=\lbrack 21^3,(1^2)^5\rbrack\), \(\varkappa_1{\mathfrak{G}}=(1,0^5)\),
of the first excited state of the TKT a.\(2\) as the search pattern
for seeking three-stage towers of \(5\)-class fields was the following.
Firstly, since the derived subgroup \(\mathfrak{G}^\prime\) for the ground state of TKT a.\(2\) and a.\(3\) is of type \((1^2)\),
a result of Blackburn ensures a two-stage tower with \(\ell_5{K}=2\) for these \(347\) cases.
Secondly, the first excited state of the TKT a.\(3\) does not admit a unique candidate
for the  second \(5\)-class groups \(\mathfrak{G}=\mathrm{G}_5^2{K}\),
and finally, the \(13\) cases of \(5\)-groups with TKT a.\(1\) form a subgraph with considerable complexity
of the coclass tree \(\mathcal{T}^1{R}\) with root \(R:=C_5\times C_5\).
Therefore, we arrived at the following group theoretic results.


\begin{proposition}
\label{prp:Second5ClassGroupGT}

Up to isomorphism, there exists a unique finite metabelian \(5\)-group \(\mathfrak{G}\) such that

\begin{equation}
\label{eqn:AP5}
\tau_0{\mathfrak{G}}=(1^2), \quad \tau_1{\mathfrak{G}}=\lbrack 21^3,(1^2)^5\rbrack, \quad \text{ and } \quad
\varkappa_1{\mathfrak{G}}=(1,0^5).
\end{equation}

\end{proposition}


\begin{theorem}
\label{thm:Second5ClassGroupGT}

The unique metabelian \(5\)-group \(\mathfrak{G}\) in Proposition
\ref{prp:Second5ClassGroupGT}
is of order \(\lvert\mathfrak{G}\rvert=5^6=15\,625\), nilpotency class \(\mathrm{cl}(\mathfrak{G})=5\), coclass \(\mathrm{cc}(\mathfrak{G})=1\),
and is isomorphic to \(\langle 5^6,635\rangle\) in the SmallGroups Library
\cite{BEO1,BEO2}.
It is a terminal vertex (leaf) on branch \(\mathcal{B}(5)\) of the coclass tree \(\mathcal{T}^1{R}\)
with abelian root \(R:=\langle 5^2,2\rangle\) of type \((1^2)\).
The group \(\mathfrak{G}\) has relation rank \(d_2{\mathfrak{G}}=4\) and its derived subgroup \(\mathfrak{G}^\prime\) is abelian of type \((1^4)\).

\end{theorem}


\begin{proposition}
\label{prp:5ClassTowerGroupGT}

Up to isomorphism,
there exist precisely five pairwise non-isomorphic finite non-metabelian \(5\)-groups \(H_1,\ldots,H_5\)
whose second derived quotient \(H_j/H_j^{\prime\prime}\) is isomorphic to the group \(\mathfrak{G}\) of Theorem
\ref{thm:Second5ClassGroupGT},
for \(1\le j\le 5\).

\end{proposition}


\begin{theorem}
\label{thm:5ClassTowerGroupGT}

The five non-metabelian \(5\)-groups \(H_j\) in Proposition
\ref{prp:5ClassTowerGroupGT}
are of order \(\lvert H_j\rvert=5^7=78\,125\), nilpotency class \(\mathrm{cl}(H_j)=5\), coclass \(\mathrm{cc}(H_j)=2\)
and derived length \(\mathrm{dl}(H_j)=3\), for \(1\le j\le 5\).
They are isomorphic to \(\langle 5^7,n\rangle\) with \(n\in\lbrace 361,373,374,385,386\rbrace\) in the SmallGroups Library
\cite{BEO2},
located as terminal vertices (leaves) on the descendant tree \(\mathcal{T}(R)\),
but not on the coclass tree \(\mathcal{T}^1{R}\), of the abelian root \(R=\langle 5^2,2\rangle\);
in fact, they are sporadic and do not belong to any coclass tree.
Their second derived subgroup \(H_j^{\prime\prime}\) is cyclic of order \(5\),
and is contained in the centre \(\zeta_1{H_j}\) of type \((1^2)\),
i.e., each \(H_j\) is centre by metabelian.
The groups \(H_j\) have relation rank \(d_2{H_j}=3\) and the abelianization \(H_j^\prime/H_j^{\prime\prime}\)
of their derived subgroup \(H_j^\prime\) is of type \((1^4)\).

\end{theorem}

\begin{proof}
The detailed proof of Proposition
\ref{prp:Second5ClassGroupGT},
Theorem
\ref{thm:Second5ClassGroupGT},
Proposition
\ref{prp:5ClassTowerGroupGT},
and Theorem
\ref{thm:5ClassTowerGroupGT}
is conducted in
\cite[\S\ 4]{Ma12}.
A diagram of the pruned descendant tree \(\mathcal{T}(R)\) with root \(R=\langle 25,2\rangle\),
where the finite \(5\)-groups \(\mathfrak{G}\) and \(H_j\) are located,
is shown in
\cite[\S\ 7, Fig. 1]{Ma12}.
Polycyclic power commutator presentations of the groups \(H_j\) are given in
\cite[\S\ 7]{Ma12}
and a diagram of their normal lattice,
including the lower and upper central series,
is drawn in
\cite[\S\ 7, Fig. 2]{Ma12}.
\end{proof}


Now we come to the number theoretic harvest of the group theoretic results
by translating back to arithmetic in the manner of \S\
\ref{ss:ClassExtAndTransfer}.
Here, we exceptionally use the power form of abelian type invariants,
and we dispense with formal exponents denoting iteration.

\begin{theorem}
\label{thm:MainTheorem5}

Let \(K\) be a real quadratic field with \(5\)-class group \(\mathrm{Cl}_5{K}\) of type \(\lbrack 5,5\rbrack\)
and denote by \(L_1,\ldots,L_6\) its six unramified cyclic quintic extensions.
If \(K\) possesses the \(5\)-capitulation type

\begin{equation}
\label{eqn:TKT5}
\varkappa(K)\sim (1,0,0,0,0,0), \text{ with fixed point } 1,
\end{equation}

\noindent
in the six extensions \(L_i\), and if the \(5\)-class groups \(\mathrm{Cl}_5{L_i}\)
are given by

\begin{equation}
\label{eqn:TTT5}
\tau(K)\sim\left(\lbrack 25,5,5,5\rbrack,\lbrack 5,5\rbrack,\lbrack 5,5\rbrack,\lbrack 5,5\rbrack,\lbrack 5,5\rbrack,\lbrack 5,5\rbrack\right),
\end{equation}

\noindent
then the \(5\)-class tower
\(K<\mathrm{F}_5^1{K}<\mathrm{F}_5^2{K}<\mathrm{F}_5^3{K}=\mathrm{F}_5^\infty{K}\)
of \(K\) has exact length \(\ell_5{K}=3\).

\end{theorem}


\begin{corollary}
\label{cor:MainTheorem5}

A real quadratic field \(K\) which satisfies the assumptions in Theorem
\ref{thm:MainTheorem5},
in particular the Formulas
\eqref{eqn:TKT5}
and
\eqref{eqn:TTT5},
has the unique second \(5\)-class group

\begin{equation}
\label{eqn:Second5ClassGroup}
\mathrm{G}_5^2{K}\simeq\langle 15625,635\rangle
\end{equation}

\noindent
with order \(5^6\), class \(5\), coclass \(1\), derived length \(2\), and relation rank \(4\),\\
and one of the following five candidates for the \(5\)-class tower group

\begin{equation}
\label{eqn:5TowerGroup}
\mathrm{G}_5^\infty{K}=\mathrm{G}_5^3{K}\simeq\langle 78125,n\rangle, \quad \text{ where } n\in\lbrace 361,373,374,385,386\rbrace,
\end{equation}

\noindent
with order \(5^7\), class \(5\), coclass \(2\), derived length \(3\), and relation rank \(3\).

\end{corollary}

\begin{proof}
For the proof of Theorem
\ref{thm:MainTheorem5}
and Corollary
\ref{cor:MainTheorem5},
the methods of \S\S\
\ref{ss:RelationRank}
and
\ref{ss:Cover}
come into the play.
The cover of the metabelian \(5\)-group \(\mathfrak{G}=\mathrm{G}_5^2{K}\simeq\langle 5^6,635\rangle\) consists of the six elements
\(\mathrm{cov}(\mathfrak{G})=\lbrace\mathfrak{G},H_1,\ldots,H_5\rbrace\),
but since the relation rank of the metabelian group is too big,
the Shafarevich cover of \(\mathfrak{G}\) with respect to any real quadratic number field \(K\) with \(\varrho=2\) reduces to
\(\mathrm{cov}(\mathfrak{G},K)=\lbrace H_1,\ldots,H_5\rbrace\),
as explained in the proofs of
\cite[Thm. 6.1]{Ma12}
and
\cite[Thm. 6.3]{Ma12}.
\end{proof}


\begin{example}
\label{exm:Realization5}

The minimal fundamental discriminant \(d\) of a real quadratic field \(K=\mathbb{Q}(\sqrt{d})\)
satisfying the conditions
\eqref{eqn:TKT5}
and
\eqref{eqn:TTT5}
is given by

\begin{equation}
\label{eqn:Realization5}
d=3\,812\,377=991\cdot 3\,847.
\end{equation}

\noindent
The next occurrence is \(d=19\,621\,905=3\cdot 5\cdot 307\cdot 4\,261\), which is currently the biggest known example,
whereas \(d=21\,281\,673=3\cdot 7\cdot 53\cdot 19\,121\) satisfies
\eqref{eqn:TTT5}
but has a different \(\varkappa(K)\sim (2,0,0,0,0,0)\), without fixed point.

\end{example}



\section{Three-stage towers of \(3\)-class fields}
\label{s:ThreeStage3Towers}

Since we have devoted the preceding \S\
\ref{s:ThreeStage5Towers}
to a detailed explanation of the general way
from a number theoretic experiment,
which prescribes a certain Artin pattern,
over the group theoretic interpretation of data and the identification of suitable groups,
to the final arithmetical statement of a criterion for three-stage towers,
we can restrict ourselves to number theoretic end results, in the sequel.

The situation in \S\
\ref{s:ThreeStage5Towers}
gives rise to a \textit{finite} cover with \(\#\mathrm{cov}(\mathfrak{G})=6\)
and a Shafarevich cover \(\mathrm{cov}(\mathfrak{G},K)\), with respect to a real quadratic field \(K\),
all of whose members \(H\) have the same derived length \(\mathrm{dl}(H)=3\),
which we shall call \textit{homogeneous}.
These considerations will be the guiding principle for a subdivision of the following results.



\subsection{Finite homogeneous Shafarevich cover}
\label{ss:FiniteHomo}

The finiteness of the cover \(\mathrm{cov}(\mathfrak{G})\) of descendants \(\mathfrak{G}\)
either of \(\langle 3^5,6\rangle\) with types c.\(18\), E.\(6\), E.\(14\)
or of \(\langle 3^5,8\rangle\) with types c.\(21\), E.\(8\), E.\(9\)
has been proven up to a certain nilpotency class \(c=\mathrm{cc}(\mathfrak{G})\) in
\cite{Ma6}
for section E, and in
\cite{Ma10}
for section c.
In fact, the cardinality of the cover is expected to increase linearly with \(c\), for sections E and c.
We present an examplary result with TKT \(\varkappa\) of type E.\(9\),
where a \textit{complex} quadratic base field \(K\)
compels a homogeneous Shafarevich cover \(\mathrm{cov}(\mathfrak{G},K)\)
of its second \(3\)-class group \(\mathfrak{G}=\mathrm{G}_3^2{K}\).


\begin{theorem}
\label{thm:FiniteHomo}

Let \(K\) be a complex quadratic field with \(3\)-class group \(\mathrm{Cl}_3{K}\) of type \((1^2)\)
and denote by \(L_1,\ldots,L_4\) its four unramified cyclic cubic extensions.
If \(K\) possesses the \(3\)-capitulation type

\begin{equation}
\label{eqn:TKT3FiniteHomo}
\varkappa(K)\sim (2334), \text{ with two fixed points } 3,4 \text{ and without a transposition},
\end{equation}

\noindent
in the four extensions \(L_i\), and if the \(3\)-class groups \(\mathrm{Cl}_3{L_i}\)
are given by

\begin{equation}
\label{eqn:TTT3FiniteHomo}
\tau(K)\sim\left(21,(j+1,j),21,21\right), \text{ for some } 2\le j\le 6,
\end{equation}

\noindent
then the \(3\)-class tower
\(K<\mathrm{F}_3^1{K}<\mathrm{F}_3^2{K}<\mathrm{F}_3^3{K}=\mathrm{F}_3^\infty{K}\)
of \(K\) has exact length \(\ell_3{K}=3\).

\end{theorem}


\begin{corollary}
\label{cor:FiniteHomo}

A complex quadratic field \(K\) which satisfies the assumptions in Theorem
\ref{thm:FiniteHomo},
in particular the Formulas
\eqref{eqn:TKT3FiniteHomo}
and
\eqref{eqn:TTT3FiniteHomo} with \(2\le j\le 6\),
has the second \(3\)-class group

\begin{equation}
\label{eqn:Second3ClassGroupHomo}
\mathrm{G}_3^2{K}\simeq\langle 3^6,54\rangle(-\#1;1-\#1;1)^{j-2}-\#1;n, \quad \text{ where } n\in\lbrace 4,6\rbrace,
\end{equation}

\noindent
with order \(3^{2j+3}\), class \(2j+1\), coclass \(2\), derived length \(2\), and relation rank \(3\),\\
and the \(3\)-class tower group

\begin{equation}
\label{eqn:3TowerGroupHomo}
\mathrm{G}_3^\infty{K}=\mathrm{G}_3^3{K}\simeq\langle 3^6,54\rangle(-\#2;1-\#1;1)^{j-2}-\#2;n, \quad \text{ where } n\in\lbrace 4,6\rbrace,
\end{equation}

\noindent
with order \(3^{3j+2}\), class \(2j+1\), coclass \(j+1\), derived length \(3\), and relation rank \(2\).

\end{corollary}

\begin{proof}
The statements of Theorem
\ref{thm:FiniteHomo}
and Corollary
\ref{cor:FiniteHomo}
arise by specialization from the more general
\cite[Thm. 6.1 and Cor. 6.1, pp. 751--753]{Ma8}.
Here, we restrict to the single TKT E.\(9\)
and to the smaller range \(2\le j\le 6\),
which has been realized by explicit numerical results already,
as shown in Example
\ref{exm:RealizationFiniteHomo}.
\end{proof}


\begin{example}
\label{exm:RealizationFiniteHomo}

The fundamental discriminant \(d\) of complex quadratic fields \(K=\mathbb{Q}(\sqrt{d})\),
with minimal absolute value \(\lvert d\rvert\),
satisfying the conditions
\eqref{eqn:TKT3FiniteHomo}
and
\eqref{eqn:TTT3FiniteHomo},
with increasing values of the parameter \(2\le j\le 6\),
corresponding to excited states of increasing order of type E.\(9\),
are given by

\begin{equation}
\label{eqn:RealizationFiniteHomo}
d=-9\,748,-297\,079,-1\,088\,808,-11\,091\,140,-99\,880\,548.
\end{equation}

\noindent
The associated second \(3\)-class groups \(\mathfrak{G}=\mathrm{G}_3^2{K}\)
form a periodic sequence of vertices on the coclass tree \(\mathcal{T}^2{\langle 3^5,8\rangle}\)
drawn in the diagram
\cite[Fig. 4, p. 755]{Ma8}.
The smallest example \(d=-9\,748\) with \(j=2\),
corresponding to the ground state of type E.\(9\),
was the object of our joint investigations with Bush in
\cite[Cor. 4.1.1, p. 775]{BuMa}.

\end{example}



\subsection{Finite heterogeneous Shafarevich cover}
\label{ss:FiniteHetero}

In contrast to the previous \S\
\ref{ss:FiniteHomo},
a \textit{real} quadratic base field \(K\) with TKT \(\varkappa\) of type E.\(9\)
is not able to enforce a homogeneous Shafarevich cover
of its second \(3\)-class group \(\mathfrak{G}=\mathrm{G}_3^2{K}\).
In this situation,
\(\mathrm{cov}(\mathfrak{G},K)\) contains elements \(H\) of derived lengths \(2\le\mathrm{dl}(H)\le 3\),
and there arises the necessity to establish criteria
for distinguishing between two- and three-stage towers.
According to Theorem
\ref{thm:RstrAPofCompleteCover},
the (simple) IPAD of first order, \(\tau^{(1)}{K}=\lbrack\tau_0{K};\tau_1{K}\rbrack\),
which forms the first order approximation of the layered TTT \(\tau(K)=\lbrack\tau_0{K};\ldots;\tau_v{K}\rbrack\),
is unable to admit a decision,
and we have to proceed to abelian type invariants of second order.
The computation of the iterated IPAD of second order, \(\tau^{(2)}{K}=\lbrack\tau_0{K};(\tau^{(1)}{L})_{L\in\mathrm{Lyr}_1{K}}\rbrack\),
where \(\tau^{(1)}{L}=\lbrack\tau_0{L};\tau_1{L}\rbrack=\lbrack\tau_0{L};(\tau_0{M})_{M\in\mathrm{Lyr}_1{L}}\rbrack\),
for each \(L\in\mathrm{Lyr}_1{K}\),
requires the construction of unramified abelian and non-abelian extensions \(M\vert K\)
of relative degree \(\lbrack M:K\rbrack=3^2\),
that is, of absolute degree \(\lbrack M:\mathbb{Q}\rbrack=18\),
whereas in \S\
\ref{ss:FiniteHomo},
cyclic extensions \(L\vert K\) of absolute degree \(\lbrack L:\mathbb{Q}\rbrack=6\) were sufficient.
Due to the complexity of the scenario, we now prefer a restriction to the ground state.


\begin{theorem}
\label{thm:FiniteHetero}

Let \(K\) be a real quadratic field with \(3\)-class group \(\mathrm{Cl}_3{K}\) of type \((1^2)\)
and denote by \(L_1,\ldots,L_4\) its four unramified cyclic cubic extensions.
If \(K\) possesses the \(3\)-capitulation type

\begin{equation}
\label{eqn:TKT3FiniteHetero}
\varkappa(K)\sim (2334), \text{ with two fixed points } 3,4 \text{ and without a transposition},
\end{equation}

\noindent
in the four extensions \(L_i\), and if the \(3\)-class groups \(\mathrm{Cl}_3{L_i}\)
are given by

\begin{equation}
\label{eqn:TTT3FiniteHetero}
\tau(K)\sim\left(21,32,21,21\right),
\end{equation}

\noindent
then each \(L_i\) has four unramified cyclic cubic extensions \(M_{i,1},\ldots,M_{i,4}\),
which are also unramified but not necessarily abelian over \(K\),
and the \(3\)-class groups \(\mathrm{Cl}_3{M_{i,j}}\) admit the following decision
about the length \(\ell_3{K}\) of the \(3\)-class tower of \(K\):

\begin{equation}
\label{eqn:2ndIPAD3FiniteHetero2}
\text{if } \tau(L_i)\sim\left(2^21,21,21,21\right), \text{ for } i\in\lbrace 1,3,4\rbrace,
\end{equation}

\noindent
then the \(3\)-class tower
\(K<\mathrm{F}_3^1{K}<\mathrm{F}_3^2{K}=\mathrm{F}_3^\infty{K}\)
of \(K\) has exact length \(\ell_3{K}=2\),

\begin{equation}
\label{eqn:2ndIPAD3FiniteHetero3}
\text{but if } \tau(L_i)\sim\left(2^21,31,31,31\right), \text{ for } i\in\lbrace 1,3,4\rbrace,
\end{equation}

\noindent
then the \(3\)-class tower
\(K<\mathrm{F}_3^1{K}<\mathrm{F}_3^2{K}<\mathrm{F}_3^3{K}=\mathrm{F}_3^\infty{K}\)
of \(K\) has exact length \(\ell_3{K}=3\).

\end{theorem}


\begin{corollary}
\label{cor:FiniteHetero}

A real quadratic field \(K\) which satisfies the assumptions in Theorem
\ref{thm:FiniteHomo},
in particular the Formulas
\eqref{eqn:TKT3FiniteHetero}
and
\eqref{eqn:TTT3FiniteHetero},
has the second \(3\)-class group

\begin{equation}
\label{eqn:Second3ClassGroupHetero}
\mathrm{G}_3^2{K}\simeq\langle 3^7,n\rangle, \quad \text{ where } n\in\lbrace 302,306\rbrace,
\end{equation}

\noindent
with order \(2\,187\), class \(5\), coclass \(2\), derived length \(2\), and relation rank \(3\),\\
and, if Formula
\eqref{eqn:2ndIPAD3FiniteHetero3}
is satisfied, the \(3\)-class tower group

\begin{equation}
\label{eqn:3TowerGroupHetero}
\mathrm{G}_3^\infty{K}=\mathrm{G}_3^3{K}\simeq\langle 3^6,54\rangle-\#2;n, \quad \text{ where } n\in\lbrace 2,6\rbrace,
\end{equation}

\noindent
with order \(6\,561\), class \(5\), coclass \(3\), derived length \(3\), and relation rank \(2\),\\
otherwise the \(3\)-class tower group coincides with the second \(3\)-class group.

\end{corollary}

\begin{proof}
The statements of Theorem
\ref{thm:FiniteHetero}
and Corollary
\ref{cor:FiniteHetero}
have been proved in
\cite[Thm. 6.3, pp. 298--299]{Ma7}
and
\cite[Thm. 4.2]{Ma11}.
We point out that the remaining component
\(\tau(L_2)\sim\left(2^21,31^2,31^2,31^2\right)\)
of the iterated IPAD of second order
does not admit a decision,
and that the common entry \(\mathrm{Cl}_3{M_{i,1}}\simeq (2^21)\) of all components \(\tau(L_i)\)
corresponds to the Hilbert \(3\)-class field \(\mathrm{F}_3^1{K}\) of \(K\),
whereas all other extensions \(M_{i,j}\) with \(j>1\) are non-abelian over \(K\).
\end{proof}


\begin{example}
\label{exm:RealizationFiniteHetero}

The fundamental discriminants \(0<d<10^7\) of real quadratic fields \(K=\mathbb{Q}(\sqrt{d})\),
satisfying the conditions
\eqref{eqn:TKT3FiniteHetero},
\eqref{eqn:TTT3FiniteHetero},
and
\eqref{eqn:2ndIPAD3FiniteHetero3},
resp.
\eqref{eqn:2ndIPAD3FiniteHetero2},
are given by

\begin{equation}
\label{eqn:RealizationFiniteHetero3}
d=342\,664,1\,452\,185,1\,787\,945,4\,861\,720,5\,976\,988,8\,079\,101,9\,674\,841,
\end{equation}

\noindent
resp.

\begin{equation}
\label{eqn:RealizationFiniteHetero2}
d=4\,760\,877,6\,652\,929,7\,358\,937,9\,129\,480.
\end{equation}

\noindent
Evidence of a similar behaviour of real quadratic fields \(K\) with
the first excited state of one of the TKTs \(\varkappa(K)\) in section E is provided in
\cite[Example 4.1]{Ma13}.

\end{example}



\subsection{Infinite cover}
\label{ss:Infinite}

The infinitude of the cover, \(\#\mathrm{cov}(\mathfrak{G})=\infty\),
has been proven by Bartholdi and Bush
\cite{BaBu}
for the sporadic \(3\)-group \(\mathfrak{G}=N:=\langle 3^6,45\rangle\),
with \(\tau=\lbrack 1^3,1^3,21,1^3\rbrack\) and \(\varkappa=(4443)\) of type H.\(4^\ast\),
and it is conjectured for \(\mathfrak{G}=W:=\langle 3^6,57\rangle\),
with \(\tau=\lbrack(21)^4\rbrack\) and \(\varkappa=(2143)\) of type G.\(19^\ast\).

In the former case,
\(\mathrm{cov}(N)\) contains an infinite sequence of Schur \(\sigma\)-groups.
Consequently, even the Shafarevich cover \(\mathrm{cov}(N,K)\)
with respect to complex quadratic fields \(K\) of \(3\)-class rank \(\varrho=2\) is infinite.
Furthermore, it may be called \textit{infinitely heterogeneous}
in the sense of unbounded derived length.
This fact causes the considerable difficulty that iterated IPADs of increasing order
are required for the distinction between the members of \(\mathrm{cov}(N,K)\).
Already for separating the leading two members,
we have to compute extensions \(M\vert K\) of absolute degree \(\lbrack M:\mathbb{Q}\rbrack=54\)
in the third layer over \(K\),
as the following theorem shows.


\begin{theorem}
\label{thm:Infinite}

Let \(K\) be a complex quadratic field with \(3\)-class group \(\mathrm{Cl}_3{K}\) of type \((1^2)\)
and denote by \(L_1,\ldots,L_4\) its four unramified cyclic cubic extensions.
If \(K\) possesses the \(3\)-capitulation type

\begin{equation}
\label{eqn:TKT3Infinite}
\varkappa(K)\sim (4443), \text{ nearly constant, without fixed points, and containing a transposition},
\end{equation}

\noindent
in the four extensions \(L_i\), and if the \(3\)-class groups \(\mathrm{Cl}_3{L_i}\)
are given by

\begin{equation}
\label{eqn:TTT3Infinite}
\tau(K)\sim\left(1^3,1^3,21,1^3\right),
\end{equation}

\noindent
then \(L_i\) has thirteen unramified bicyclic bicubic extensions \(M_{i,1},\ldots,M_{i,13}\), for \(i\in\lbrace 1,2,4\rbrace\),
but \(L_3\) has only four unramified abelian extensions \(M_{3,1},\ldots,M_{3,4}\) of relative degree \(3^2\),
which are also unramified but not necessarily abelian over \(K\),
and the \(3\)-class groups \(\mathrm{Cl}_3{M_{i,j}}\) admit the following decision
about the length \(\ell_3{K}\) of the \(3\)-class tower of \(K\):

\begin{equation}
\label{eqn:2ndIPAD3Infinite3}
\text{if } \tau_2(L_i)\sim\left(2^21,(21^2)^{12}\right), \text{ for } i\in\lbrace 1,2\rbrace,
\tau_2(L_3)\sim\left(2^21,(2^2)^3\right), \tau_2(L_4)\sim\left(2^21,(1^3)^3,(2^2)^3,(21)^6\right),
\end{equation}

\noindent
then the \(3\)-class tower
\(K<\mathrm{F}_3^1{K}<\mathrm{F}_3^2{K}<\mathrm{F}_3^3{K}=\mathrm{F}_3^\infty{K}\)
of \(K\) has exact length \(\ell_3{K}=3\),

\begin{equation}
\label{eqn:2ndIPAD3Infinite4}
\text{but if } \tau_2(L_i)\sim\left((2^21)^4,(31^2)^9\right), \text{ for } i\in\lbrace 1,2\rbrace,
\tau_2(L_3)\sim\left(2^21,(32)^3\right), \tau_2(L_4)\sim\left(2^21,(1^3)^3,(32)^3,(21)^6\right),
\end{equation}

\noindent
then the \(3\)-class tower
\(K<\mathrm{F}_3^1{K}<\mathrm{F}_3^2{K}<\mathrm{F}_3^3{K}\le\ldots\le\mathrm{F}_3^\infty{K}\)
of \(K\) may have any length \(\ell_3{K}\ge 3\).

\end{theorem}


\begin{corollary}
\label{cor:Infinite}

A complex quadratic field \(K\) which satisfies the assumptions in Theorem
\ref{thm:Infinite},
in particular the Formulas
\eqref{eqn:TKT3Infinite}
and
\eqref{eqn:TTT3Infinite},
has the second \(3\)-class group

\begin{equation}
\label{eqn:Second3ClassGroupInfinite}
\mathrm{G}_3^2{K}\simeq\langle 3^6,45\rangle,
\end{equation}

\noindent
with order \(729\), class \(4\), coclass \(2\), derived length \(2\), and relation rank \(4\),\\
and, if Formula
\eqref{eqn:2ndIPAD3Infinite3}
is satisfied, the \(3\)-class tower group

\begin{equation}
\label{eqn:3TowerGroupInfinite}
\mathrm{G}_3^\infty{K}=\mathrm{G}_3^3{K}\simeq\langle 3^6,45\rangle-\#2;2,
\end{equation}

\noindent
with order \(6\,561\), class \(5\), coclass \(3\), derived length \(3\), and relation rank \(2\),\\
otherwise the \(3\)-class tower group is of order at least \(3^{11}\) and may be any of the Schur \(\sigma\)-groups
\(\langle 3^6,45\rangle(-\#2;1-\#1;2)^j-\#2;2\) with \(j\ge 1\) and derived length at least \(3\).

\end{corollary}

\begin{proof}
The statements of Theorem
\ref{thm:Infinite}
and Corollary
\ref{cor:Infinite}
have been proved in
\cite[Thm. 6.5, pp. 304--306]{Ma7}
and
\cite[\S\ 4.4, Tbl. 4]{Ma11}.
We point out that the first layer
\(\tau_1(L_i)\sim\left((21^2)^4,(2^2)^9\right)\), for \(1\le i\le 2\),
\(\tau_1(L_3)\sim\left(21^2,(31)^3\right)\),
\(\tau_1(L_4)\sim\left((21^2)^4,(1^2)^9\right)\),
of the iterated IPAD of second order
does not admit a decision.
\end{proof}


\begin{example}
\label{exm:RealizationInfinite}

The fundamental discriminants \(-30\,000<d<0\) of complex quadratic fields \(K=\mathbb{Q}(\sqrt{d})\),
satisfying the conditions
\eqref{eqn:TKT3Infinite},
\eqref{eqn:TTT3Infinite},
and
\eqref{eqn:2ndIPAD3Infinite3},
resp.
\eqref{eqn:2ndIPAD3Infinite4},
are given by

\begin{equation}
\label{eqn:RealizationInfinite3}
d=-3\,896,-25\,447,-27\,355,
\end{equation}

\noindent
resp.

\begin{equation}
\label{eqn:RealizationInfinite4}
d=-6\,583,-23\,428,-27\,991.
\end{equation}

\end{example}



\section{Three-stage towers of \(2\)-class fields}
\label{s:ThreeStage2Towers}

For historical reasons,
the very first discovery of a \(p\)-class tower with three stages for \(p=2\) merits attention.
It is due to Bush
\cite{Bu}
in \(2003\).
He investigated complex quadratic fields \(K\) with \(2\)-class rank \(\varrho=2\),
since it is relatively easy to compute the unramified \(2\)-extensions \(M\vert K\) in several layers
with absolute degrees \(\lbrack M:\mathbb{Q}\rbrack\in\lbrace 4,8,16\rbrace\).


\begin{theorem}
\label{thm:ThreeStage2Towers}

Let \(K\) be a complex quadratic field with \(2\)-class group \(\mathrm{Cl}_2{K}\) of type \((21)\),
denote by \(L_{1,1},\ldots,L_{1,3}\) its three unramified quadratic extensions,
and by \(L_{2,1},\ldots,L_{2,3}\) its three unramified abelian quartic extensions, such that
\(L_{2,3}=\prod_{i=1}^3\,L_{1,i}\) is bicyclic biquadratic and \(L_{1,3}=\bigcap_{i=1}^3\,L_{2,i}\).
If the \(2\)-class groups \(\mathrm{Cl}_2{L_{1,i}}\), resp. \(\mathrm{Cl}_2{L_{2,i}}\), are given by

\begin{equation}
\label{eqn:TTT2Lyr1and2}
\tau_1(K)=\left(2^2,31,1^3\right), \quad \tau_2(K)=\left(2^2,21^2,21^2\right),
\end{equation}

\noindent
then \(L_{1,i}\) has three unramified quadratic extensions \(M_{i,1},\ldots,M_{i,3}\) for \(1\le i\le 2\),
and \(L_{1,3}\) has seven unramified quadratic extensions \(M_{3,1},\ldots,M_{3,7}\),
which are also unramified but not necessarily abelian over \(K\),
and the \(2\)-class groups \(\mathrm{Cl}_2{M_{i,j}}\) admit the following statement.

\begin{equation}
\label{eqn:2ndIPADLyr1}
\text{If } \tau_1(L_{1,1})=\left((21^2)^3\right), \quad \tau_1(L_{1,2})=\left(21^2,41,41\right), \quad \tau_1(L_{1,3})=\left((21^2)^6,2^2\right),
\end{equation}

\noindent
then the \(2\)-class tower
\(K<\mathrm{F}_2^1{K}<\mathrm{F}_2^2{K}<\mathrm{F}_2^3{K}=\mathrm{F}_2^\infty{K}\)
of \(K\) has exact length \(\ell_2{K}=3\).

\end{theorem}


\begin{corollary}
\label{cor:ThreeStage2Towers}

A complex quadratic field \(K\) which satisfies the assumptions in Theorem
\ref{thm:ThreeStage2Towers},
in particular the Formulas
\eqref{eqn:TTT2Lyr1and2}
and
\eqref{eqn:2ndIPADLyr1},
has the second \(2\)-class group

\begin{equation}
\label{eqn:Second2ClassGroup}
\mathrm{G}_2^2{K}\simeq\langle 2^7,84\rangle,
\end{equation}

\noindent
with order \(128\), class \(4\), coclass \(3\), derived length \(2\), and relation rank \(4\),\\
and one of the following two candidates for the \(2\)-class tower group

\begin{equation}
\label{eqn:2TowerGroup}
\mathrm{G}_2^\infty{K}=\mathrm{G}_2^3{K}\simeq\langle 2^8,n\rangle, \quad \text{ where } n\in\lbrace 426,427\rbrace,
\end{equation}

\noindent
with order \(256\), class \(5\), coclass \(3\), derived length \(3\), and relation rank \(3\).

\end{corollary}

\begin{proof}
The statements of Theorem
\ref{thm:ThreeStage2Towers}
and Corollary
\ref{cor:ThreeStage2Towers}
have essentially been proved by Bush in
\cite[Prop. 2, p. 321]{Bu}.
The derived subgroup of \(\langle 2^7,84\rangle\) is abelian of type \((21^2)\)
\end{proof}



\section{Tree topologies}
\label{s:TreeTopologies}

Let \(p\) be a prime,
\(n>m\ge 1\) be integers,
and \(K\) be a number field.
Then both, the \(n\)th and \(m\)th \(p\)-class group of \(K\),
are vertices of the descendant tree \(\mathcal{T}\left(\mathrm{G}_p^1{K}\right)\)
of the \(p\)-class group \(\mathrm{Cl}_p{K}=\mathrm{G}_p^1{K}\) of \(K\).
The vertex \(R:=\mathrm{G}_p^1{K}\) is the \textit{abelian tree root}.
Several invariants describe the mutual location of the vertices
\(\mathrm{G}_p^n{K}\) and \(\mathrm{G}_p^m{K}\)
in the tree topology.

\begin{definition}
\label{dfn:Deltas}

By the \textit{class increment}, resp. \textit{coclass increment}, we understand the difference
\(\Delta\mathrm{cl}(n,m):=\mathrm{cl}(\mathrm{G}_p^n{K})-\mathrm{cl}(\mathrm{G}_p^m{K})\),
resp. \(\Delta\mathrm{cc}(n,m):=\mathrm{cc}(\mathrm{G}_p^n{K})-\mathrm{cc}(\mathrm{G}_p^m{K})\).
The biggest common ancestor of \(\mathrm{G}_p^m{K}\) and \(\mathrm{G}_p^n{K}\)
is called their \textit{fork}, denoted by \(\mathrm{Fork}(m,n)\).

\end{definition}

\begin{remark}
\label{rmk:Deltas}

If we define the \textit{logarithmic order} of a finite \(p\)-group \(G\) by
\(\mathrm{lo}(G):=\log_p{\mathrm{ord}(G)}\),
and consider the situation in Definition
\ref{dfn:Deltas},
then the \textit{log ord increment},
\(\Delta\mathrm{lo}(n,m):=\mathrm{lo}(\mathrm{G}_p^n{K})-\mathrm{lo}(\mathrm{G}_p^m{K})\),
satisfies the relation
\(\Delta\mathrm{lo}(n,m)=\Delta\mathrm{cl}(n,m)+\Delta\mathrm{cc}(n,m)\).

\end{remark}

The concepts actually make sense for \(p\)-class towers of length \(\ell_p{K}\ge 3\).
For two-stage towers, we have the following trivial fact.


\begin{proposition}
\label{prp:RootTopology}

For any \(n\ge 2\), we have
\(\Delta\mathrm{cl}(n,1)=\mathrm{cl}(\mathrm{G}_p^n{K})-1\),
and \(\mathrm{Fork}(1,n)\) is given by the abelian tree root \(R=\mathrm{G}_p^1{K}\).

\end{proposition}


\renewcommand{\arraystretch}{1.1}

\begin{table}[ht]
\caption{Simple tree topologies of three-stage \(p\)-class towers}
\label{tbl:SimpleTreeTopologies}
\begin{center}
\begin{tabular}{|c|r|l|c|c|c|l|c|c|c|c|}
\hline
 \(p\) & \(d\) & TKT & \(\Delta\mathrm{lo}\) & \(\Delta\mathrm{cc}\) & \(\Delta\mathrm{cl}\) & \(\mathrm{Fork}(2,3)\) & Topology & \(G^{\prime\prime}\) & Proof & Ref. \\
\hline
 \(2\) &       \(-1\,780\) & B.\(8\)               &  \(1\) &  \(0\) &  \(1\) & Y\(=\mathfrak{G}=\pi{G}\)            &   child &         \((1)\) & \(\tau^{(2)}\) & \cite{Bu} \\
\hline
 \(3\) &      \(957\,013\) & H.\(4^\ast\)          &  \(1\) &  \(0\) &  \(1\) & N\(=\mathfrak{G}=\pi{G}\)            &   child &         \((1)\) & \(\tau^{(2)}\) & \cite{Ma7} \\
 \(3\) &      \(214\,712\) & G.\(19^\ast\)         &  \(1\) &  \(0\) &  \(1\) & W\(=\mathfrak{G}=\pi{G}\)            &   child &         \((1)\) & \(\tau^{(2)}\) & \cite{Ma11} \\
 \(3\) &  \(21\,974\,161\) & G.\(19^\ast\uparrow\) &  \(4\) &  \(1\) &  \(3\) & W\(=\mathfrak{G}=\pi^3{G}\)          & descent &   \((1,1,1,1)\) & \(\tau^{(2)}\) & \cite{Ma11} \\
\hline
 \(3\) &       \(-3\,896\) & H.\(4^\ast\)          &  \(2\) &  \(1\) &  \(1\) & N\(=\mathfrak{G}=\pi{G}\)            & bastard &       \((1,1)\) & \(\tau_\ast^{(2)}\) & \cite{Ma7} \\
 \(3\) &     ? \(-6\,896\) & H.\(4^\ast\uparrow\)  &  \(5\) &  \(2\) &  \(3\) & N\(=\mathfrak{G}=\pi^3{G}\)          & descent &     \((2,2,1)\) & \(\tau^{(3)}\) & \cite{Ma11} \\
\hline
 \(3\) &      \(534\,824\) & c.\(18\)              &  \(1\) &  \(0\) &  \(1\) & Q\(=\mathfrak{G}=\pi{G}\)            &   child &         \((1)\) & \(\tau^{(2)}\) & \cite{Ma10} \\
 \(3\) &   \(1\,030\,117\) & c.\(18\)              &  \(1\) &  \(0\) &  \(1\) & Q\(=\mathfrak{G}=\pi{G}\)            &   child &         \((1)\) & \(\tau^{(2)}\) & \cite{Ma10} \\
 \(3\) &  \(13\,714\,789\) & c.\(18\uparrow\)      &  \(1\) &  \(0\) &  \(1\) & Q\({}_8=\mathfrak{G}=\pi{G}\)        &   child &         \((1)\) & \(\tau^{(2)}\) & \cite{Ma10} \\
 \(3\) & \(241\,798\,776\) & c.\(18\uparrow^2\)    &  \(1\) &  \(0\) &  \(1\) & Q\({}_{10}=\mathfrak{G}=\pi{G}\)     &   child &         \((1)\) & \(\tau^{(2)}\) & \cite{Ma10} \\
 \(3\) &      \(540\,365\) & c.\(21\)              &  \(1\) &  \(0\) &  \(1\) & U\(=\mathfrak{G}=\pi{G}\)            &   child &         \((1)\) & \(\tau^{(2)}\) & \cite{Ma10} \\
 \(3\) &   \(1\,001\,957\) & c.\(21\uparrow\)      &  \(1\) &  \(0\) &  \(1\) & U\({}_8=\mathfrak{G}=\pi{G}\)        &   child &         \((1)\) & \(\tau^{(2)}\) & \cite{Ma10} \\
 \(3\) & \(407\,086\,012\) & c.\(21\uparrow^2\)    &  \(1\) &  \(0\) &  \(1\) & U\({}_{10}=\mathfrak{G}=\pi{G}\)     &   child &         \((1)\) & \(\tau^{(2)}\) & \cite{Ma10} \\
\hline
\end{tabular}
\end{center}
\end{table}


In Table
\ref{tbl:SimpleTreeTopologies}
and
\ref{tbl:AdvancedTreeTopologies},
we summarize all tree topologies currently known for
three-stage \(p\)-class towers of quadratic base fields \(K=\mathbb{Q}(\sqrt{d})\)
with fundamental discriminant \(d\).
We put \(n=3\) and \(m=2\), and use the abbreviations
\(\Delta\mathrm{lo}:=\Delta\mathrm{lo}(3,2)\),
\(\Delta\mathrm{cl}:=\Delta\mathrm{cl}(3,2)\),
\(\Delta\mathrm{cc}:=\Delta\mathrm{cc}(3,2)\).
Forks are labelled with Ascione's identifiers
\cite{AHL,As1},
\(\mathrm{N}:=\langle 3^6,45\rangle\),
\(\mathrm{Q}:=\langle 3^6,49\rangle\),
\(\mathrm{U}:=\langle 3^6,54\rangle\),
\(\mathrm{W}:=\langle 3^6,57\rangle\),
avoiding the long symbols in angle brackets of the SmallGroups Library
\cite{BEO1,BEO2}.
Additionally, we define two ad hoc-identifiers
\(\mathrm{Y}:=\langle 2^7,84\rangle\), resp.
\(\mathrm{Z}:=\langle 5^5,30\rangle\),
for \(p=2\), resp. \(p=5\).
Four vertices on the mainline containing \(\mathrm{Q}\), resp. \(\mathrm{U}\), are denoted by
\(\mathrm{Q}_8:=\langle 3^7,285\rangle-\#1;1\),
\(\mathrm{Q}_{10}:=\langle 3^7,285\rangle(-\#1;1)^3\),
resp.
\(\mathrm{U}_8:=\langle 3^7,303\rangle-\#1;1\),
\(\mathrm{U}_{10}:=\langle 3^7,303\rangle(-\#1;1)^3\),
using relative ANUPQ identifiers
\cite{GNO}.


A question mark in front of a discriminant \(d\) indicates that
the result is conjectural only.
By \(\mathrm{P}_7\) we denote the vertex \(\langle 3^7,64\rangle\).



\begin{figure}[ht]
\caption{Tree topology of type parent -- child}
\label{fig:TreeTopo1}


\setlength{\unitlength}{1.0cm}
\begin{picture}(6,5)(0,-4)


\put(0,0){\line(0,-1){2}}
\multiput(-0.1,0)(0,-2){2}{\line(1,0){0.2}}

\put(-0.2,-0){\makebox(0,0)[rc]{\(729\)}}
\put(0.2,-0){\makebox(0,0)[lc]{\(3^6\)}}
\put(-0.2,-2){\makebox(0,0)[rc]{\(2\,187\)}}
\put(0.2,-2){\makebox(0,0)[lc]{\(3^7\)}}

\put(0,-2){\vector(0,-1){1}}
\put(0,-3.5){\makebox(0,0)[ct]{Order \(3^n\)}}


\put(3.8,0.4){\makebox(0,0)[rc]{\(\mathfrak{G}=\pi{G}=\langle 45\rangle\)}}
\put(4.2,0.4){\makebox(0,0)[lc]{parent}}
\put(4,0){\circle*{0.2}}

\put(4,0){\line(0,-1){2}}

\put(3.9,-2.1){\framebox(0.2,0.2){}}
\put(3.8,-2.4){\makebox(0,0)[rc]{\(G=\langle 273\rangle\)}}
\put(4.2,-2.4){\makebox(0,0)[lc]{child}}


\end{picture}

\end{figure}
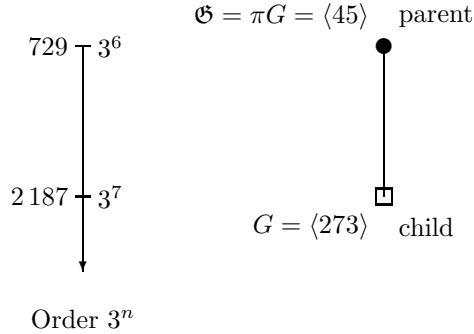

The diagram in Figure
\ref{fig:TreeTopo1}
visualizes the \textit{simple child topology}
of the mutual location between the second and third \(3\)-class group,
\(\mathfrak{G}=\mathrm{Gal}(\mathrm{F}_3^2{K}\vert K)\) and \(G=\mathrm{Gal}(\mathrm{F}_3^3{K}\vert K)\),
where \(K=\mathbb{Q}(\sqrt{d})\) is the real quadratic field with discriminant \(d=957\,013\).
The non-metabelian \(3\)-group \(G\) is a child, that is an immediate descendant of step size \(s=1\),
of the metabelian \(3\)-group \(\mathfrak{G}\).



\begin{figure}[ht]
\caption{Tree topology of type parent -- bastard}
\label{fig:TreeTopo1b}


\setlength{\unitlength}{1.0cm}
\begin{picture}(8,7)(0,-6)


\put(0,0){\line(0,-1){4}}
\multiput(-0.1,0)(0,-2){3}{\line(1,0){0.2}}

\put(-0.2,-0){\makebox(0,0)[rc]{\(729\)}}
\put(0.2,-0){\makebox(0,0)[lc]{\(3^6\)}}
\put(-0.2,-2){\makebox(0,0)[rc]{\(2\,187\)}}
\put(0.2,-2){\makebox(0,0)[lc]{\(3^7\)}}
\put(-0.2,-4){\makebox(0,0)[rc]{\(6\,561\)}}
\put(0.2,-4){\makebox(0,0)[lc]{\(3^8\)}}

\put(0,-4){\vector(0,-1){1}}
\put(0,-5.5){\makebox(0,0)[ct]{Order \(3^n\)}}


\put(3.8,0.4){\makebox(0,0)[rc]{\(\mathfrak{G}=\pi{G}=\langle 45\rangle\)}}
\put(4.2,0.4){\makebox(0,0)[lc]{parent}}
\put(4,0){\circle*{0.2}}

\put(4,0){\line(1,-2){2}}

\put(5.9,-4.1){\framebox(0.2,0.2){}}
\put(5.8,-4.4){\makebox(0,0)[rc]{\(G=\langle 45\rangle-\#2;2\)}}
\put(6.2,-4.4){\makebox(0,0)[lc]{bastard}}


\end{picture}

\end{figure}
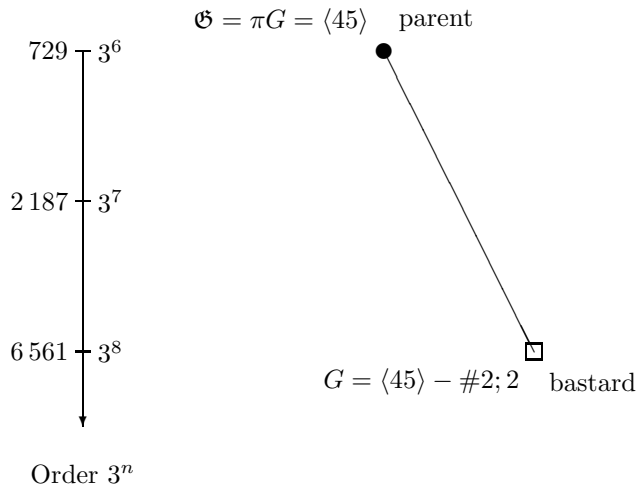

The diagram in Figure
\ref{fig:TreeTopo1b}
visualizes the \textit{simple bastard topology}
of the mutual location between the second and third \(3\)-class group,
\(\mathfrak{G}=\mathrm{Gal}(\mathrm{F}_3^2{K}\vert K)\) and \(G=\mathrm{Gal}(\mathrm{F}_3^3{K}\vert K)\),
where \(K=\mathbb{Q}(\sqrt{d})\) is the complex quadratic field with discriminant \(d=-3\,896\).
The non-metabelian \(3\)-group \(G\) is a bastard, that is an immediate descendant of step size \(s=2\),
of the metabelian \(3\)-group \(\mathfrak{G}\).



\renewcommand{\arraystretch}{1.1}

\begin{table}[ht]
\caption{Advanced tree topologies of three-stage \(p\)-class towers}
\label{tbl:AdvancedTreeTopologies}
\begin{center}
\begin{tabular}{|c|r|l|c|c|c|l|c|c|c|c|}
\hline
 \(p\) & \(d\) & TKT & \(\Delta\mathrm{lo}\) & \(\Delta\mathrm{cc}\) & \(\Delta\mathrm{cl}\) & \(\mathrm{Fork}(2,3)\) & Topology & \(G^{\prime\prime}\) & Proof & Ref. \\
\hline
 \(5\) &   \(3\,812\,377\) & a.\(2\)               &  \(1\) &  \(1\) &  \(0\) & Z\(=\pi\mathfrak{G}=\pi{G}\)         & sibling &         \((1)\) & \(\tau^{(1)}\) & \cite{Ma12} \\
\hline
 \(3\) &       \(-9\,748\) & E.\(9\)               &  \(1\) &  \(1\) &  \(0\) & U\(=\pi\mathfrak{G}=\pi{G}\)         & sibling &         \((1)\) & \(\tau^{(1)}\) & \cite{BuMa} \\
 \(3\) &     \(-297\,079\) & E.\(9\uparrow\)       &  \(2\) &  \(2\) &  \(0\) & U\(=\pi^3\mathfrak{G}=\pi^3{G}\)     &    fork &         \((2)\) & \(\tau^{(1)}\) & \cite{Ma8} \\
 \(3\) &  \(-1\,088\,808\) & E.\(9\uparrow^2\)     &  \(3\) &  \(3\) &  \(0\) & U\(=\pi^5\mathfrak{G}=\pi^5{G}\)     &    fork &         \((3)\) & \(\tau^{(1)}\) & \cite{Ma8} \\
 \(3\) & \(-11\,091\,140\) & E.\(9\uparrow^3\)     &  \(4\) &  \(4\) &  \(0\) & U\(=\pi^7\mathfrak{G}=\pi^7{G}\)     &    fork &         \((4)\) & \(\tau^{(1)}\) & \cite{Ma8} \\
 \(3\) & \(-94\,880\,548\) & E.\(9\uparrow^4\)     &  \(5\) &  \(5\) &  \(0\) & U\(=\pi^9\mathfrak{G}=\pi^9{G}\)     &    fork &         \((5)\) & \(\tau^{(1)}\) & \cite{Ma8} \\
\hline
 \(3\) &  \(14\,252\,156\) & c.\(18\uparrow\)      &  \(2\) &  \(1\) &  \(1\) & Q\(=\pi^2\mathfrak{G}=\pi^3{G}\)     &    fork &       \((2)\) & \(\tau^{(2)}\) & \cite{Ma10} \\
 \(3\) & \(174\,458\,681\) & c.\(18\uparrow^2\)    &  \(3\) &  \(2\) &  \(1\) & Q\(=\pi^4\mathfrak{G}=\pi^5{G}\)     &    fork &       \((3)\) & \(\tau^{(2)}\) & \cite{Ma10} \\
 \(3\) &  \(25\,283\,701\) & c.\(21\uparrow\)      &  \(2\) &  \(1\) &  \(1\) & U\(=\pi^2\mathfrak{G}=\pi^3{G}\)     &    fork &       \((2)\) & \(\tau^{(2)}\) & \cite{Ma10} \\
 \(3\) & \(116\,043\,324\) & c.\(21\uparrow^2\)    &  \(3\) &  \(2\) &  \(1\) & U\(=\pi^4\mathfrak{G}=\pi^5{G}\)     &    fork &       \((3)\) & \(\tau^{(2)}\) & \cite{Ma10} \\
\hline
 \(3\) & ? \(1\,535\,117\) & d.\(23\)              &  \(2\) &  \(2\) &  \(0\) & P\({}_7=\pi\mathfrak{G}=\pi{G}\)     & sibling &       \((1,1)\) & \(\tau^{(2)}\) & \\
\hline
 \(3\) &   ? \(-124\,363\) & F.\(7\ast\)           & \(11\) &  \(7\) &  \(4\) & P\({}_7=\pi\mathfrak{G}=\pi^5{G}\)   &    fork & \((3^2,2,1^3)\) & \(\tau^{(2)}\) & \\
 \(3\) &   ? \(-469\,787\) & F.\(11\)              & \(18\) & \(12\) &  \(6\) & P\({}_7=\pi^3\mathfrak{G}=\pi^9{G}\) &    fork & \((5,3^3,2^2)\) & \(\tau^{(2)}\) & \\
\hline
\end{tabular}
\end{center}
\end{table}



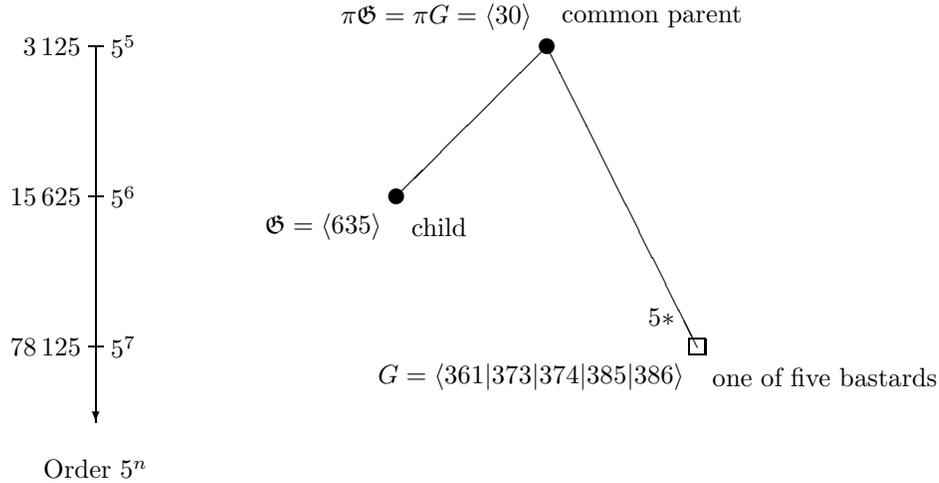
\begin{figure}[ht]
\caption{Tree topology of type siblings}
\label{fig:TreeTopo2}


\setlength{\unitlength}{1.0cm}
\begin{picture}(10,7)(0,-6)


\put(0,0){\line(0,-1){4}}
\multiput(-0.1,0)(0,-2){3}{\line(1,0){0.2}}

\put(-0.2,0){\makebox(0,0)[rc]{\(3\,125\)}}
\put(0.2,0){\makebox(0,0)[lc]{\(5^5\)}}
\put(-0.2,-2){\makebox(0,0)[rc]{\(15\,625\)}}
\put(0.2,-2){\makebox(0,0)[lc]{\(5^6\)}}
\put(-0.2,-4){\makebox(0,0)[rc]{\(78\,125\)}}
\put(0.2,-4){\makebox(0,0)[lc]{\(5^7\)}}

\put(0,-4){\vector(0,-1){1}}
\put(0,-5.5){\makebox(0,0)[ct]{Order \(5^n\)}}


\put(5.8,0.4){\makebox(0,0)[rc]{\(\pi{\mathfrak{G}}=\pi{G}=\langle 30\rangle\)}}
\put(6.2,0.4){\makebox(0,0)[lc]{common parent}}
\put(6,0){\circle*{0.2}}

\put(6,0){\line(-1,-1){2}}

\put(4,-2){\circle*{0.2}}
\put(3.8,-2.4){\makebox(0,0)[rc]{\(\mathfrak{G}=\langle 635\rangle\)}}
\put(4.2,-2.4){\makebox(0,0)[lc]{child}}

\put(6,0){\line(1,-2){2}}

\put(7.7,-3.6){\makebox(0,0)[rc]{\(5\ast\)}}
\put(7.9,-4.1){\framebox(0.2,0.2){}}
\put(7.8,-4.4){\makebox(0,0)[rc]{\(G=\langle 361\vert 373\vert 374\vert 385\vert 386\rangle\)}}
\put(8.2,-4.4){\makebox(0,0)[lc]{one of five bastards}}


\end{picture}

\end{figure}

The diagram in Figure
\ref{fig:TreeTopo2}
visualizes the \textit{advanced siblings topology}
of the mutual location between the second and third \(5\)-class group,
\(\mathfrak{G}=\mathrm{Gal}(\mathrm{F}_5^2{K}\vert K)\) and \(G=\mathrm{Gal}(\mathrm{F}_5^3{K}\vert K)\),
where \(K=\mathbb{Q}(\sqrt{d})\) is the real quadratic field with discriminant \(d=3\,812\,377\).



\section{Biquadratic base fields containing the \(p\)th roots of unity}
\label{s:BicyclicBiquadratic}

\subsection{Dirichlet fields}
\label{ss:Dirichlet}

The first examples of fields, where a violation of the Shafarevich Theorem
\ref{thm:RelationRank}
in its misprinted version
\cite[Thm. 6, \((18^\prime)\)]{Sh}
occurred, have been found by Azizi, Zekhnini and Taous
\cite{AZT},
the violation itself has been recognized by ourselves.

A bicyclic biquadratic field
\(k=\mathbb{Q}(\sqrt{-1},\sqrt{d})\) with squarefree radicand \(d>1\)
is totally complex with signature \((r_1,r_2)=(0,2)\).
Thus, the torsionfree Dirichlet unit rank of \(k\) is \(r=r_1+r_2-1=1\).
The particular fields with radicand \(d=p_1p_2q\),
where \(p_1\equiv 1\pmod{8}\), \(p_2\equiv 5\pmod{8}\) and
\(q\equiv 3\pmod{4}\) are prime numbers such that
\(\left(\frac{p_1}{p_2}\right)=-1\), \(\left(\frac{p_1}{q}\right)=-1\) and \(\left(\frac{p_2}{q}\right)=-1\),
have a \(2\)-class group \(\mathrm{Cl}_2{k}\) of type \((1^3)\)
\cite{AZT},
and a \(2\)-class tower of length \(\ell_2{k}=2\)
\cite{AZT}.
Thus, the \(2\)-class rank of \(k\) is \(\varrho=3\), and,
since \(k\) trivially contains the second roots of unity, we have the invariant \(\theta=1\),
with respect to the even prime \(p=2\).
In
\cite{Ma8},
we have identified the possible \(2\)-class tower groups \(G=\mathrm{G}_2^\infty{k}\) of \(k\) as
\(\langle 32,35\rangle\), \(\langle 64,181\rangle\), \(\langle 128,984\rangle\), etc.,
visualized in the diagram
\cite[Fig. 2, p. 752]{Ma8}.
A computation with the aid of MAGMA
\cite{MAGMA}
shows that these metabelian \(2\)-groups all have the maximal admissible relation rank \(d_2{G}=5\),
in accordance with our corrected Formula
\eqref{eqn:RelationRank}
in Theorem
\ref{thm:RelationRank},
\(d_2(G)\le d_1{G}+r+\theta=\varrho+r+\theta=3+1+1=5\),
whereas the misprinted formula
\cite[Thm. 6, \((18^\prime)\)]{Sh}
yields the contradiction \(5=d_2{G}\le d_1{G}+1=\varrho+1=3+1=4\).

In contrast, no violation could be found in our previous joint paper
\cite{AZTM},
where the possible \(2\)-class tower groups
\(G=\mathrm{G}_2^\infty{k}\in
\lbrace\langle 64,180\rangle,\langle 128,985\vert 986\rangle,\langle 256,6720\vert 6721\rangle,\ldots\rbrace\),
visualized in the diagram
\cite[Fig. 5, p. 1208]{AZTM},
all have relation rank \(d_2{G}=4\) only.



\subsection{Eisenstein fields}
\label{ss:Eisenstein}

However, another series of violations showed up in our joint paper
\cite{ATTDM}
on bicyclic biquadratic fields
\(k=\mathbb{Q}(\sqrt{-3},\sqrt{d})\) with squarefree radicand \(d>1\)
and \(3\)-class group \(\mathrm{Cl}_3{k}\) of type \((1^2)\),
which also have torsionfree Dirichlet unit rank \(r=1\),
but \(3\)-class rank \(\varrho=2\) only.
Due to the inclusion of \(\sqrt{-3}\in k\),
the fields contain the third roots of unity, and the invariant \(\theta\) takes the value \(1\),
with respect to the odd prime \(p=3\).

In \S\ 7,
Example 7.2 of
\cite{ATTDM},
we have seen that
among the \(930\) suitable values of the radicand in the range \(0<d<50\,000\),
there occur \(197\) (\(\approx 21.2\%\), e.g. \(d=469\))
with second \(3\)-class group \(\mathfrak{G}:=\mathrm{G}_3^2{k}\simeq\langle 81,9\rangle\)
and \(42\) (\(\approx 4.5\%\), e.g. \(d=7\,453\))
with \(\mathfrak{G}\simeq\langle 729,95\rangle\).
These \(3\)-groups are of coclass  \(\mathrm{cc}(\mathfrak{G})=1\) and have relation rank \(d_2{\mathfrak{G}}=4\),
as a computation by means of MAGMA
\cite{MAGMA}
shows.
Since their cover \(\mathrm{cov}(\mathfrak{G})=\lbrace\mathfrak{G}\rbrace\) is trivial,
which means that there does not exist a finite non-metabelian \(3\)-group \(H\)
of derived length \(\mathrm{dl}(H)\ge 3\)
such that \(\mathfrak{G}\) is isomorphic to the second derived quotient \(H/H^{\prime\prime}\),
the second \(3\)-class group \(\mathfrak{G}\) must coincide with the \(3\)-class tower group \(G:=\mathrm{G}_3^\infty{k}\) already.

The misprinted original version of the Theorem by Shafarevich
\cite{Sh}
(Teorema 6, p. 83, in the Russian original,
resp. Theorem 6, formula (\(18^\prime\)), p. 140, in the English translation)
enforces the relation rank \(d_2{\mathfrak{G}}=d_2{G}\le d_1{G}+1=\varrho+1=2+1=3\),
which is obviously a contradiction to our result \(d_2{\mathfrak{G}}=4\)
for more than a quarter (\(25.7\%\)) of all fields \(k\) under investigation.

Fortunately, our corrected Formula
\eqref{eqn:RelationRank}
in Theorem
\ref{thm:RelationRank},
\(d_2(G)\le d_1{G}+r+\theta=\varrho+r+\theta=2+1+1=4\),
is in accordance with the fact
that these metabelian \(3\)-groups have the maximal admissible relation rank \(d_2{\mathfrak{G}}=4\). 

Our Theorem 7.1 in
\cite{ATTDM}
is the first indication of the misprint in
\cite{Sh}
for an odd prime \(p=3\).

Since all second \(3\)-class groups \(\mathfrak{G}=\mathrm{G}_3^2{k}\)
in Theorem 8.3 and Theorem 8.7 of
\cite{ATTDM}
have relation rank \(d_2{\mathfrak{G}}=5\),
they cannot coincide with the \(3\)-tower group \(G=\mathrm{G}_3^\infty{k}\),
and the corresponding \(3\)-class field tower must have length \(\ell_3{k}\) at least \(3\)
whenever the coclass is \(\mathrm{cc}(\mathfrak{G})\ge 2\).



\section{Two-stage towers of \(p\)-class fields}
\label{s:TwoStageTowers}

The \(p\)-groups in the stem of Hall's isoclinism family \(\Phi_6\)
\cite{Hl}
are two-generated metabelian groups \(G=\langle x,y\rangle\)
of order \(\lvert G\rvert=p^5\), nilpotency class \(\mathrm{cl}(G)=3\) and coclass \(\mathrm{cc}(G)=2\).
They do not exist for \(p=2\), but for odd prime numbers \(p\ge 3\)
they uniformly arise as descendants of step size \(2\)
of the extra special group \(G_0^3(0,0)\) of order \(p^3\) and exponent \(p\)
\cite[Tbl. 1, p. 483]{Ma2}
and thus form top vertices of the coclass graph \(\mathcal{G}(p,2)\).
The reason for this behaviour is the nuclear rank \(\nu\) of \(G_0^3(0,0)\),
which is only \(\nu=1\) for \(p=2\), but \(\nu=2\) for \(p\ge 3\)
giving rise to a bifurcation from \(\mathcal{G}(p,1)\) to \(\mathcal{G}(p,2)\).

The basic properties of the stem groups in \(\Phi_6\) have been discussed in
\cite[\S\ 3.5, pp. 445--451]{Ma4},
where we pointed out that the groups for \(p=3\) are irregular in the sense of Hall,
but all groups for \(p\ge 5\) are regular, since \(\mathrm{cl}(G)=3<p\).
In
\cite[pp. 1--10]{Ma5},
we used commutator calculus for determining the transfer kernel type \(\varkappa(G)\)
in the regular case \(p\ge 5\), for the first time.
The regularity admits the simplification that all transfers uniformly map to \(p\)th powers.
A summary of the systematic results is as follows.

\begin{theorem}
\label{thm:TrfKerStem5Icl6}
(D. C. Mayer, \(08\) November \(2010\))\\

The transfer kernel type (TKT) \(\varkappa(G)\) of the
\(12\) top vertices \(G\) with abelianization \(G/G^\prime\simeq (1^2)\) of the coclass graph \(\mathcal{G}(5,2)\),
which form the stem \(\Phi_6(0)\) of \(5\)-groups in Hall's isoclinism family \(\Phi_6\),
is given by Table
\ref{tbl:TrfKerStem5Icl6}.
A partial characterization is given by the counter \(\eta\)
of fixed point transfer kernels \(\kappa(i)=\mathrm{A}\),
resp. abelianizations (transfer targets) \(\tau(i)\) of type \((1^3)\),
\[\eta=\#\lbrace 1\le i\le 6\mid\kappa(i)=\mathrm{A}\rbrace
=\#\lbrace 1\le i\le 6\mid\tau(i)=(1^3)\rbrace.\]
An asterisk after the SmallGroup identifier
\cite{BEO2}
denotes a Schur \(\sigma\)-group.

\end{theorem}

\begin{proof}
This is Theorem 2.2 in
\cite{Ma5}.
The characterization by the counter \(\eta\) is due to
Theorem 3.8 and Table 3.3 of
\cite[\S\ 3.2.5, pp. 423--424]{Ma4}.
\end{proof}

In Table
\ref{tbl:TrfKerStem5Icl6},
each \(5\)-group is identified primarily with the symbol given by James
\cite{Jm},
who used Hall's isoclinism families and regular type invariants
\cite{Hl},
and Easterfield's characterization of maximal subgroups
\cite{Ef}.
The property is an invariant characterization of the TKT,
whereas the multiplet \(\varkappa\) depends on the selection of generators and
on the numeration of maximal subgroups.



\renewcommand{\arraystretch}{1.1}

\begin{table}[ht]
\caption{TKT of twelve \(5\)-groups of order \(5^5\) in the isoclinism family \(\Phi_6\)}
\label{tbl:TrfKerStem5Icl6}
\begin{center}
\begin{tabular}{|ll|cccc|}
\hline
 \multicolumn{2}{|c|}{Identifier of the \(5\)-group}  & \multicolumn{4}{|c|}{Transfer kernel type (TKT)}               \\
 James                  & SmallGroup                  & \(\eta\) & \(\varkappa\) & Cycle pattern  & Property           \\
\hline
 \(\Phi_6(2^21)_a\)     & \(\langle 3125,14\rangle\ast\) & 6 & \((123456)\) & \((1)(2)(3)(4)(5)(6)\) & identity           \\
 \(\Phi_6(2^21)_{b_1}\) & \(\langle 3125,11\rangle\ast\) & 2 & \((125364)\) & \((1)(2)(3564)\)       & \(4\)-cycle        \\
 \(\Phi_6(2^21)_{b_2}\) & \(\langle 3125,7\rangle\)   & 2 & \((126543)\) & \((1)(2)(36)(45)\)     & two \(2\)-cycles   \\
 \(\Phi_6(2^21)_{c_1}\) & \(\langle 3125,8\rangle\ast\)  & 1 & \((612435)\) & \((16532)(4)\)         & \(5\)-cycle        \\
 \(\Phi_6(2^21)_{c_2}\) & \(\langle 3125,13\rangle\ast\) & 1 & \((612435)\) & \((16532)(4)\)         & \(5\)-cycle        \\
 \(\Phi_6(2^21)_{d_0}\) & \(\langle 3125,10\rangle\)  & 0 & \((214365)\) & \((12)(34)(56)\)       & three \(2\)-cycles \\
 \(\Phi_6(2^21)_{d_1}\) & \(\langle 3125,12\rangle\ast\) & 0 & \((512643)\) & \((154632)\)           & \(6\)-cycle        \\
 \(\Phi_6(2^21)_{d_2}\) & \(\langle 3125,9\rangle\ast\)  & 0 & \((312564)\) & \((132)(456)\)         & two \(3\)-cycles   \\
\hline
 \(\Phi_6(21^3)_a\)     & \(\langle 3125,4\rangle\)   & 2 & \((022222)\) &                        &nrl. const. with fp.\\
 \(\Phi_6(21^3)_{b_1}\) & \(\langle 3125,5\rangle\)   & 1 & \((011111)\) &                        & nearly constant    \\
 \(\Phi_6(21^3)_{b_2}\) & \(\langle 3125,6\rangle\)   & 1 & \((011111)\) &                        & nearly constant    \\
 \(\Phi_6(1^5)\)        & \(\langle 3125,3\rangle\)   & 6 & \((000000)\) &                        & constant           \\
\hline
\end{tabular}
\end{center}
\end{table}



\begin{theorem}
\label{thm:TrfKerStem7Icl6}
(D. C. Mayer, \(15\) October \(2012\))\\

The transfer kernel type (TKT) \(\varkappa(G)\) of the
\(14\) top vertices \(G\) with abelianization \(G/G^\prime\simeq (1^2)\) of the coclass graph \(\mathcal{G}(7,2)\),
which form the stem \(\Phi_6(0)\) of \(7\)-groups in Hall's isoclinism family \(\Phi_6\),
is given by Table
\ref{tbl:TrfKerStem7Icl6}.
A partial characterization is given by the counter \(\eta\)
of fixed point transfer kernels \(\kappa(i)=\mathrm{A}\),
resp. abelianizations (transfer targets) \(\tau(i)\) of type \((1^3)\),
\[\eta=\#\lbrace 1\le i\le 8\mid\kappa(i)=\mathrm{A}\rbrace
=\#\lbrace 1\le i\le 8\mid\tau(i)=(1^3)\rbrace.\]
A star after the SmallGroup identifier
\cite{BEO2}
denotes a Schur \(\sigma\)-group.

\end{theorem}

\begin{proof}
This is Theorem 2.3 in
\cite{Ma5}.
The characterization by the counter \(\eta\) is due to
Theorem 3.8 and Table 3.3 of
\cite[\S\ 3.2.5, pp. 423--424]{Ma4}
\end{proof}

In Table
\ref{tbl:TrfKerStem7Icl6},
each \(7\)-group is identified primarily with the symbol given by James
\cite{Jm},
who used Hall's isoclinism families and regular type invariants
\cite{Hl},
and Easterfield's characterization of maximal subgroups
\cite{Ef}.
The property is an invariant characterization of the TKT,
whereas the multiplet \(\varkappa\) depends on the selection of generators and
on the numeration of maximal subgroups.



\renewcommand{\arraystretch}{1.1}

\begin{table}[ht]
\caption{TKT of fourteen \(7\)-groups of order \(7^5\) in the isoclinism family \(\Phi_6\)}
\label{tbl:TrfKerStem7Icl6}
\begin{center}
\begin{tabular}{|ll|cccc|}
\hline
 \multicolumn{2}{|c|}{Identifier of the \(7\)-group}   & \multicolumn{4}{|c|}{Transfer kernel type (TKT)}                 \\
 James                  & SmallGroup                   & \(\eta\) & \(\varkappa\) & Cycle pattern           & Property    \\
\hline
 \(\Phi_6(2^21)_a\)     & \(\langle 16807,7\rangle\ast\)  & 8 & \((12345678)\) & \((1)(2)(3)(4)(5)(6)(7)(8)\)  & identity    \\
 \(\Phi_6(2^21)_{b_1}\) & \(\langle 16807,11\rangle\ast\) & 2 & \((12753864)\) & \((1)(2)(376845)\)     & \(6\)-cycle        \\
 \(\Phi_6(2^21)_{b_2}\) & \(\langle 16807,12\rangle\ast\) & 2 & \((12637485)\) & \((1)(2)(364)(578)\)   & two \(3\)-cycles   \\
 \(\Phi_6(2^21)_{b_3}\) & \(\langle 16807,13\rangle\)  & 2 & \((12876543)\) & \((1)(2)(38)(47)(56)\) & three transpos.    \\
 \(\Phi_6(2^21)_{c_1}\) & \(\langle 16807,9\rangle\ast\)  & 1 & \((61247583)\) & \((1657832)(4)\)       & \(7\)-cycle        \\
 \(\Phi_6(2^21)_{c_3}\) & \(\langle 16807,8\rangle\ast\)  & 1 & \((61247583)\) & \((1657832)(4)\)       & \(7\)-cycle        \\
 \(\Phi_6(2^21)_{d_0}\) & \(\langle 16807,10\rangle\)  & 0 & \((21573846)\) & \((12)(35)(47)(68)\)   & four transpos.     \\
 \(\Phi_6(2^21)_{d_1}\) & \(\langle 16807,15\rangle\ast\) & 0 & \((41238756)\) & \((1432)(5867)\)       & two \(4\)-cycles   \\
 \(\Phi_6(2^21)_{d_2}\) & \(\langle 16807,16\rangle\ast\) & 0 & \((81256347)\) & \((18745632)\)         & \(8\)-cycle        \\
 \(\Phi_6(2^21)_{d_3}\) & \(\langle 16807,14\rangle\ast\) & 0 & \((71283465)\) & \((17648532)\)         & \(8\)-cycle        \\
\hline
 \(\Phi_6(21^3)_a\)     & \(\langle 16807,4\rangle\)   & 2 & \((02222222)\) &                        &nrl. const. with fp.\\
 \(\Phi_6(21^3)_{b_1}\) & \(\langle 16807,6\rangle\)   & 1 & \((01111111)\) &                        & nearly constant    \\
 \(\Phi_6(21^3)_{b_3}\) & \(\langle 16807,5\rangle\)   & 1 & \((01111111)\) &                        & nearly constant    \\
 \(\Phi_6(1^5)\)        & \(\langle 16807,3\rangle\)   & 8 & \((00000000)\) &                        & constant           \\
\hline
\end{tabular}
\end{center}
\end{table}



For all isomorphism classes of \(p\)-groups
in the stem \(\Phi_6(0)\) of the isoclinism family \(\Phi_6\),
some information on descendants and on the TKT
can be provided independently of the prime \(p\ge 5\).

\begin{theorem}
\label{thm:StemIcl6}

(D. C. Mayer and M. F. Newman)

The descendant tree \(\mathcal{T}(G)\) and
the transfer kernel type \(\varkappa(G)\) of the \(p\)-groups \(G\)
in the stem of \(\Phi_6\), that is, \(p+7\) isomorphism classes of groups,
can be described in the following uniform way,
for any odd prime \(p\ge 5\),
where \(\nu\) denotes the smallest positive quadratic non-residue modulo \(p\).

\begin{enumerate}

\item
The first \(4\) groups are infinitely capable vertices \(G\) of the coclass graph \(\mathcal{G}(p,2)\)
giving rise to infinite coclass trees \(\mathcal{T}^2{G}\) of descendants.
Their TKT \(\varkappa(G)\) is nearly constant and
contains at least one total transfer, indicated by a zero, \(\varkappa_i=0\).

\begin{enumerate}
\item
\(\Phi_6(1^5)\) has constant TKT \((\overbrace{0,\ldots,0}^{p+1 \text{ times}})\), entirely consisting of total transfers.
\item
\(\Phi_6(21^3)_a\) has nearly constant TKT \((0,\overbrace{2,\ldots,2}^{p \text{ times}})\), with fixed point \(\varkappa_2=2\).
\item
\(\Phi_6(21^3)_{b_r}\) has nearly constant TKT \((0,\overbrace{1,\ldots,1}^{p \text{ times}})\), without fixed point,
for \(r\in\lbrace 1,\nu\rbrace\).
\end{enumerate}

\item
The next \(2\) groups are finitely capable vertices \(G\) of the coclass graph \(\mathcal{G}(p,2)\)
giving rise to finite trees of descendants within this graph.
However, they give rise to infinitely many descendants
spread over higher coclass graphs \(\mathcal{G}(p,r)\), \(r\ge 3\).
Their TKT \(\varkappa(G)\) is a permutation
whose cycle pattern entirely consists of transpositions.

\begin{enumerate}
\item
\(\Phi_6(2^21)_{b_{(p-1)/2}}\) has TKT \((1,2,\ldots)\) with two fixed points,
whose cycle pattern consists of \(\frac{p-1}{2}\) transpositions.
\item
\(\Phi_6(2^21)_{d_0}\) has TKT \((2,1,\ldots)\) without fixed points,
whose cycle pattern consists of \(\frac{p+1}{2}\) transpositions,
the first of them being \((1,2)\).
\end{enumerate}

\item
The last \(p+1\) groups are Schur \(\sigma\)-groups,
in particular, they are terminal vertices of the coclass graph \(\mathcal{G}(p,2)\)
without descendants.
Their TKT \(\varkappa(G)\) is a permutation
whose cycle pattern does not contain transpositions.

\begin{enumerate}
\item
\(\Phi_6(2^21)_a\) has TKT \((1,2,\ldots,p+1)\), the identity permutation
with \(p+1\) fixed points.
\item
\(\Phi_6(2^21)_{b_r}\) has TKT \((1,2,\ldots)\) with two fixed points,
whose cycle pattern consists of cycles of length \(\frac{p-1}{\mathrm{gcd}(r,p-1)}>2\),
for \(1\le r<\frac{p-1}{2}\).
\item
\(\Phi_6(2^21)_{c_r}\) has TKT \((6,1,2,4,\ldots)\) with a single fixed point
and \(\varkappa_6=5\),
for \(r\in\lbrace 1,\nu\rbrace\).
\item
\(\Phi_6(2^21)_{d_r}\) has TKT without fixed points,
for \(1\le r\le\frac{p-1}{2}\).
\end{enumerate}

\end{enumerate}

\end{theorem}

\begin{proof}
This is Theorem 2.1 in
\cite{Ma5}.
The TKT \(\varkappa\) depends on the presentations given by James
\cite{Jm}.
\end{proof}

For \(p+2\) isomorphism classes of stem groups in \(\Phi_6\),
the TKT depends on number theoretic properties of the prime \(p\ge 5\)
and we have given explicit results for \(p\in\lbrace 5,7\rbrace\) in Theorem
\ref{thm:TrfKerStem5Icl6}
and
\ref{thm:TrfKerStem7Icl6}.
For the other five, in (1)(a--c) and (3)(a), the TKT can be given uniformly for all \(p\).

Now we come to the number theoretic harvest of Theorem
\ref{thm:TrfKerStem5Icl6},
\ref{thm:TrfKerStem7Icl6},
and
\ref{thm:StemIcl6}.



\begin{theorem}
\label{thm:SchurSigma5}

Let \(K\) be an arbitrary number field
with \(5\)-class group \(\mathrm{Cl}_5{K}\) of type \((1^2)\),
whose Artin pattern \(\mathrm{AP}(K)=(\tau(K),\varkappa(K))\) is given by

\begin{enumerate}
\item
either
\(\varkappa(K)=(1,2,3,4,5,6)\) (identity) and \(\tau(K)=\left((1^3)^6\right)\)
\item
or
\(\varkappa(K)=(1,2,5,3,6,4)\) (\(4\)-cycle) and \(\tau(K)=\left((21)^4,(1^3)^2\right)\)
\item
or
\(\varkappa(K)=(6,1,2,4,3,5)\) (\(5\)-cycle) and \(\tau(K)=\left((21)^5,(1^3)\right)\)
\item
or
\(\varkappa(K)=(5,1,2,6,4,3)\) (\(6\)-cycle) and \(\tau(K)=\left((21)^6\right)\)
\item
or
\(\varkappa(K)=(3,1,2,5,6,4)\) (two \(3\)-cycles) and \(\tau(K)=\left((21)^6\right)\).
\end{enumerate}

Then \(K\) has a \(5\)-class field tower of exact length \(\ell_5{K}=2\).

\end{theorem}

\begin{proof}
According to Theorem
\ref{thm:TrfKerStem5Icl6}
and Table
\ref{tbl:TrfKerStem5Icl6},
the five given alternatives for the transfer kernel type \(\varkappa(K)\) of the number field \(K\),
which are permutations whose cycle decomposition does not contain \(2\)-cycles,
uniquely determine the Schur \(\sigma\)-groups \(\langle 5^5,n\rangle\) with identifiers \(n\in\lbrace 14,11,8,13,12,9\rbrace\),
except for the ambiguity of the \(5\)-cycle with two possibilities \(n\in\lbrace 8,13\rbrace\),
provided the second \(5\)-class group \(\mathfrak{G}=\mathrm{G}_5^2{K}\) of \(K\) belongs to the stem of \(\Phi_6\).
The latter condition is ensured by the additional assignment of the transfer target type \(\tau(K)\),
all of whose components are of logarithmic order \(\mathrm{lo}_5{\tau_i}=3\), for \(1\le i\le 6\).

It remains to show that the metabelian Schur \(\sigma\)-group \(\mathfrak{G}\)
cannot be isomorphic to the second derived quotient \(G/G^{\prime\prime}\) of a non-metabelian \(5\)-group \(G\),
that is, the cover \(\mathrm{cov}(\mathfrak{G})=\lbrace\mathfrak{G}\rbrace\) is trivial.
This is a consequence of
\cite[Lem. 4.10, p. 273]{BoEb},
which shows that an epimorphism \(G\to G/G^{\prime\prime}\simeq\mathfrak{G}\)
onto the balanced group \(\mathfrak{G}\) is an isomorphism \(G\simeq\mathfrak{G}\).
Therefore, we have \(\mathrm{G}_5^\infty{K}\simeq\mathrm{G}_5^2{K}\) and thus \(\ell_5{K}=2\).
\end{proof}



\begin{theorem}
\label{thm:SchurSigma7}

Let \(K\) be an arbitrary number field
with \(7\)-class group \(\mathrm{Cl}_7{K}\) of type \((1^2)\),
whose Artin pattern \(\mathrm{AP}(K)=(\tau(K),\varkappa(K))\) is given by

\begin{enumerate}
\item
either
\(\varkappa(K)=(1,2,3,4,5,6,7,8)\) (identity) and \(\tau(K)=\left((1^3)^8\right)\)
\item
or
\(\varkappa(K)=(1,2,7,5,3,8,6,4)\) (\(6\)-cycle) and \(\tau(K)=\left((12)^6,(1^3)^2\right)\)
\item
or
\(\varkappa(K)=(1,2,6,3,7,4,8,5)\) (two \(3\)-cycles) and \(\tau(K)=\left((12)^6,(1^3)^2\right)\)
\item
or
\(\varkappa(K)=(6,1,2,4,7,5,8,3)\) (\(7\)-cycle) and \(\tau(K)=\left((12)^7,(1^3)\right)\)
\item
or
\(\varkappa(K)=(8,1,2,5,6,3,4,7)\) (\(8\)-cycle) and \(\tau(K)=\left((12)^8\right)\)
\item
or
\(\varkappa(K)=(4,1,2,3,8,7,5,6)\) (two \(4\)-cycles) and \(\tau(K)=\left((12)^8\right)\).
\end{enumerate}

Then \(K\) has a \(7\)-class field tower of exact length \(\ell_7{K}=2\).

\end{theorem}

\begin{proof}
According to Theorem
\ref{thm:TrfKerStem7Icl6}
and Table
\ref{tbl:TrfKerStem7Icl6},
the six given alternatives for the transfer kernel type \(\varkappa(K)\) of the number field \(K\),
which are permutations whose cycle decomposition does not contain \(2\)-cycles,
uniquely determine the Schur \(\sigma\)-groups \(\langle 7^5,n\rangle\) with identifiers \(n\in\lbrace 7,11,12,9,8,15,16,14\rbrace\),
except for the ambiguity of the \(7\)-cycle with two possibilities \(n\in\lbrace 9,8\rbrace\)
and of the \(8\)-cycle with two possibilities \(n\in\lbrace 16,14\rbrace\),
provided the second \(7\)-class group \(\mathfrak{G}=\mathrm{G}_7^2{K}\) of \(K\) belongs to the stem of \(\Phi_6\).
The latter condition is ensured by the additional assignment of the transfer target type \(\tau(K)\),
all of whose components are of logarithmic order \(\mathrm{lo}_7{\tau_i}=3\), for \(1\le i\le 8\).

Similarly as in the proof of Theorem
\ref{thm:SchurSigma5},
we have \(\mathrm{G}_7^\infty{K}\simeq\mathrm{G}_7^2{K}\) and thus \(\ell_7{K}=2\).
\end{proof}

\noindent
Unfortunately, the TTT alone does not permit a decision about \(\ell_p{K}\)
with the aid of Theorem
\ref{thm:SchurSigma5}
or
\ref{thm:SchurSigma7},
in general.
The reason is that the groups
\(\langle 3125,7\rangle\) and \(\langle 3125,10\rangle\)
resp.
\(\langle 16807,10\rangle\) and \(\langle 16807,13\rangle\),
must be identified by means of their TKT.
However, an exception where the TTT is sufficient is given
in the following Corollary,
concerning \(\Phi_6(2^21)_{c_r}\) for \(r\in\lbrace 1,\nu\rbrace\)
(Theorem
\ref{thm:StemIcl6}).



\begin{corollary}
\label{cor:SchurSigma5or7}

Let \(p\ge 5\) be a prime
and let \(K\) be an arbitrary number field with
\(p\)-class group \(\mathrm{Cl}_p{K}\) of type \((1^2)\).
If the TTT of \(K\) is given by \(\tau(K)=\left((21)^p,(1^3)\right)\),
or equivalently, if \(\eta=1\),
then \(\ell_p{K}=2\).

\end{corollary}

\begin{example}
\label{exm:StemIcl6}

We succeeded in realizing all the metabelian Schur \(\sigma\)-groups in Theorem
\ref{thm:SchurSigma5}
and
\ref{thm:SchurSigma7},
with the single exception of \(\langle 16807,7\rangle\),
by second \(p\)-class groups \(\mathfrak{G}=\mathrm{G}_p^2{K}\)
of imaginary quadratic fields \(K=\mathbb{Q}(\sqrt{d})\)
with \(p\)-class tower length \(\ell_p{K}=2\),
whose absolute discriminants \(\lvert d\rvert\) are given in
tree diagrams of the top region of the coclass graph \(\mathcal{G}(5,2)\) in Figure 14 and
of \(\mathcal{G}(7,2)\) in Figure 15 of the Wikipedia article on the Artin transfer (group theory)\\
\verb+(https://en.wikipedia.org/wiki/Artin_transfer_(group_theory))+.

\end{example}

\begin{remark}
\label{rmk:StemIcl6}
We did not touch upon \(3\)-groups in the stem of \(\Phi_6\).
We only remarked that they are irregular in the sense of Hall,
which causes anomalies in the analogue of Theorem
\ref{thm:StemIcl6}
for \(p=3\):
among the \(3\)-groups in \(\Phi_6(0)\), there are,
firstly, only \(3\) instead of \(4\) infinitely capable vertices of \(\mathcal{G}(3,2)\),
namely \(\langle 243,n\rangle\) with \(n\in\lbrace 3,6,8\rbrace\),
and secondly, only \(2\) instead of \(4\) terminal Schur \(\sigma\)-groups,
namely \(\langle 243,n\rangle\) with \(n\in\lbrace 5,7\rbrace\).
Similarly as in Theorem
\ref{thm:StemIcl6},
there are \(2\) finitely capable vertices of \(\mathcal{G}(3,2)\),
namely \(\langle 243,n\rangle\) with \(n\in\lbrace 4,9\rbrace\).
An analogue of Theorem
\ref{thm:SchurSigma5}
and
\ref{thm:SchurSigma7}
for the Schur \(\sigma\)-groups \(\langle 243,n\rangle\) with \(n\in\lbrace 5,7\rbrace\)
has been proved in
\cite[Thm. 1.5, p. 407]{Ma4}.
It confirms results by Scholz and Taussky
\cite{SoTa}
with a short and elegant argumentation.
\end{remark}



\section{Infinite \(3\)-class towers}
\label{s:Infinite3Towers}

In the final section \S\ 7 of
\cite{Ma7},
we proved that the second \(3\)-class groups \(\mathfrak{G}=\mathrm{G}_3^2{K}\)
of the \(14\) complex quadratic fields \(K=\mathbb{Q}(\sqrt{d})\)
with fundamental discriminants \(-10^7<d<0\)
and \(3\)-class group \(\mathrm{Cl}_3{K}\) of type \((1^3)\)
are pairwise non-isomorphic
\cite[Thm. 7.1, p. 307]{Ma7}.
All these fields have an infinite \(3\)-class tower with \(\ell_3{K}=\infty\),
according to Koch and Venkov
\cite{KoVe}.
For the proof of this theorem in 
\cite[\S\ 7.3, p. 311]{Ma7},
the IPADs of the \(14\) fields were insufficient,
since three critical fields with discriminants
\[d\in\lbrace -4\,447\,704, -5\,067\,967, -8\,992\,363\rbrace\]
share the common accumulated (unordered) IPAD
\[\tau^{(1)}{K}=\lbrack\tau_0{K};\tau_1{K}\rbrack=\lbrack 1^3;(32^21;(21^4)^5,(2^21^2)^7)\rbrack.\]
To complete the proof
we had to use information on the occupation numbers of the accumulated (unordered) IPODs,\\
\(\varkappa_1{K}=\lbrack 1,2,6,(8)^6,9,(10)^2,13\rbrack\) with maximal occupation number \(6\) for \(d=-4\,447\,704\),\\
\(\varkappa_1{K}=\lbrack 1,2,(3)^2,(4)^2,6,(7)^2,8,(9)^2,12\rbrack\) with maximal occupation number \(2\) for \(d=-5\,067\,967\),\\
\(\varkappa_1{K}=\lbrack (2)^2,5,6,7,(9)^2,(10)^3,(12)^3\rbrack\) with maximal occupation number \(3\) for \(d=-8\,992\,363\).

In
\cite{Ma11},
we succeeded in computing the \textit{second layer} of the transfer target type, \(\tau_2{K}\),
for the critical fields
by determining the structure of the \(3\)-class groups \(\mathrm{Cl}_3{M}\)
of the \(13\) unramified bicyclic bicubic extensions \(M\vert K\) with relative degree \(\lbrack M:K\rbrack=3^2\)
and absolute degree \(18\) with the aid of the computational algebra system MAGMA
\cite{MAGMA}.
In accumulated (unordered) form, the second layer of the TTTs is given by\\
\(\tau_2{K}=\lbrack 32^51^2;4321^5;2^51^3,(3^221^5)^2;2^41^4,32^21^5;(2^21^7)^3,(2^31^5)^3\rbrack\) for \(d=-4\,447\,704\),\\
\(\tau_2{K}=\lbrack 3^22^21^4;(3^221^5)^3;32^21^5;(2^31^5)^8\rbrack\) for \(d=-5\,067\,967\), and\\
\(\tau_2{K}=\lbrack 32^21^6,(3^221^5)^3;2^41^4,32^21^5;2^21^7,(2^31^5)^6\rbrack\) for \(d=-8\,992\,363\).

These results admit incredibly powerful conclusions,
which bring us closer to the ultimate goal of determining the precise isomorphism type of \(\mathfrak{G}=\mathrm{G}_3^2{K}\).
Firstly, they clearly show that the second \(3\)-class groups of the critical fields are pairwise non-isomorphic,
without using the IPODs.
Secondly, the component with the biggest order establishes an impressively sharpened estimate
for the order of \(\mathfrak{G}=\mathrm{G}_3^2{K}\) from below.

\begin{theorem}
\label{thm:StrongEstimates}

(Fine estimates, D. C. Mayer, March \(2015\))

\noindent
None among the maximal subgroups of the second \(3\)-class group \(\mathfrak{G}=\mathrm{G}_3^2{K}\)
for the critical complex quadratic fields \(K=\mathbb{Q}(\sqrt{d})\) can be abelian.\\
The logarithmic order of \(\mathfrak{G}\) is bounded from below by\\
\(\mathrm{lo}_3{\mathfrak{G}}\ge 17\) for \(d=-4\,447\,704\),\\
\(\mathrm{lo}_3{\mathfrak{G}}\ge 16\) for \(d=-5\,067\,967\),\\
\(\mathrm{lo}_3{\mathfrak{G}}\ge 15\) for \(d=-8\,992\,363\).

\end{theorem}

\begin{proof}
This is Theorem 6.2 in
\cite{Ma11}.
\end{proof}

\begin{example}
More recently, it came to our knowledge that Leshin
\cite{Le}
has proved the infinitude \(\ell_3{N}=\infty\) of the \(3\)-class tower of
the sextic \(S_3\)-field \(N=\mathbb{Q}(\sqrt{-3},\root 3\of{D})\) with radicand
\(D=79\cdot 97=7\,663\equiv 4\pmod{9}\).
This is the normal closure of the pure cubic field \(L=\mathbb{Q}(\root 3\of{D})\)
with three ramified primes \(3,79,97\),
the last two of them congruent to \(1\) modulo \(3\) and thus split in \(\mathbb{Q}(\sqrt{-3})\).

However, the \(3\)-class group \(\mathrm{Cl}_3{N}\) is of type \((1^5)\)
with \(3\)-class rank \(\varrho=5\) and thus gives rise to
\(121\) unramified cyclic cubic extensions.
Thus, although we have seen that the search for the second \(3\)-class group \(\mathrm{G}_3^2{K}\)
of \(K=\mathbb{Q}(\sqrt{-4\,447\,704})\) with \(\varrho=3\) is very tough already, it still seems to be more promising
than the corresponding search for \(\mathrm{G}_3^2{N}\).

Finally, we remark that the pure cubic field \(L=\mathbb{Q}(\root 3\of{7\,663})\) is of the type
with a relative principal factorization in \(N\vert\mathbb{Q}(\sqrt{-3})\)
in the sense of Barrucand and Cohn,
since the \(3\)-class group \(\mathrm{Cl}_3{L}\) is of type \((1^3)\) and the class number formula
\(3^5 = h_3{N} = \frac{u}{3}\cdot h_3{L}^2 = \frac{u}{3}\cdot(3^3)^2 = u\cdot 3^5\)
yields the index \(u=1\) of the subfield units in \(N\).
\end{example}



\section{Acknowledgements}
\label{s:Acknowledgements}

The author gratefully acknowledges that his research is supported financially by the
Austrian Science Fund (FWF): P 26008-N25.

A succinct version of this article
will be presented as an invited lecture at the
First International Colloquium of Algebra, Number Theory, Cryptography and Information Security,
in Taza, Morocco, 11--12 November 2016.




\end{document}